\documentclass[a4paper,12pt,reqno]{article}
\usepackage[english]{babel}
\usepackage{amsmath,amsfonts,amssymb,amsthm,bbm}
\usepackage{graphics,epsfig,psfrag} 
\usepackage{paralist,subfigure,multirow,float}
\usepackage{mathrsfs}

\usepackage{changes}

\usepackage[sort]{natbib}   
\usepackage{hyperref} 
\usepackage[active]{srcltx} 
\usepackage{xcolor,soul} 
\usepackage[latin1]{inputenc}
\usepackage[OT1]{fontenc}

\newcommand{\commentout}[1]{}

\DeclareMathOperator{\dist}{dist}

\DeclareMathOperator\supp{supp}

\DeclareMathOperator*{\essinf}{ess{\,}inf}

\def\N{\mathbb{N}}

\def\R{\mathbb{R}}

\def\bfS{\mathbb{S}}

\let\O=\Omega
\let\e=\varepsilon
\let\vp=\varphi

\let\t=\tilde
\let\ol=\overline
\let\ul=\underline
\let\.=\cdot
\let\0=\emptyset

\let\mc=\mathcal

\def\1{\mathbbm{1}}

\def\thm#1{Theorem~\ref{thm:#1}}
\def\seq#1{(#1_n)_{n\in\N}}

\newcommand{\loc}{{\text{loc}}}
\newcommand{\Vol}{{\text{Vol}}}

\newcommand{\glSurv}{\text{glSurv}}
\newcommand{\locSurv}{\text{locSurv}}

\newcommand{\Tm}{{\mathbb{T}}}

\newcommand{\bfA}{{\bf{A}}}

\newcommand{\bfP}{{\bf{P}}}

\newcommand{\calE}{{\mathcal{E}}}

\newcommand{\Nm}{{\mathbb{N}}}

\newcommand{\bfH}{{\mathbb{H}}}
\newcommand{\bfC}{{\bf{C}}}

\newcommand{\Rm}{{\mathbb R}}

\newcommand{\Pm}{{\mathbb P}}

\newcommand{\expE}{{\mathbb E}}
\newcommand{\Ind}{{\mathbbm{1}}}

\newtheorem{theo}{Theorem}[section]
\newtheorem{lem}[theo]{Lemma}
\newtheorem{defin}[theo]{Definition}
\newtheorem{question}[theo]{Question}

\newtheorem{prop}[theo]{Proposition}
\newtheorem{cor}[theo]{Corollary}
\newtheorem{rmk}[theo]{Remark}

\newtheorem{assum}[theo]{Assumption}
\newtheorem{prob}[theo]{Problem}
\newtheorem{example}[theo]{Example}

\newcommand{\disc}{{\text{disc}}}
\newcommand{\cou}{{\text{count}}}

\long\def\metanote#1#2{{\color{#1}\
\ifmmode\hbox\fi{\sffamily\mdseries\upshape [#2]}\ }}
\long\def\OT#1{\metanote{blue!70!black}{{\tiny OT} #1}}

\newcommand{\ra}{\rightarrow}

\newcommand{\be}{\begin{equation}}
\newcommand{\ee}{\end{equation}}
\newcommand{\bi}{\begin{itemize}}
\newcommand{\ei}{\end{itemize}}

\newcommand{\eps}{\mbox{$\epsilon$}}

\newenvironment{formula}[1]{\begin{equation}\label{#1}}
{\end{equation}\noindent}

\def\Fi#1{\begin{formula}{#1}}
	\def\Ff{\end{formula}\noindent}

\newenvironment{nota}[1]{\textcolor{red}
	*\marginpar{\footnotesize{\textcolor{red}{#1}}}}

\def\BS{\color{orange}}

\newenvironment{Luca}
{\begin{center}
		\begin{tabular}{|p{0.95\textwidth}|}\hline\color{blue}}
		{ \vspace{5pt}\color{orange}\\\hline\end{tabular} 
\end{center}}

\def\pe{principal eigenvalue}
\def\pf{principal eigenfunction}

\def\SMP{strong maximum principle}

\def\L{\mathcal{L}}

\def\l{\lambda_1}

\renewcommand{\epsilon}{\varepsilon}
\newcommand{\br}{{\beta}}

\title{\bf{Sharp dimension threshold for unique FKPP steady states and BBM recurrence}}
\author{Luca Rossi\thanks{Istituto ``G.~Castelnuovo'', Sapienza Universit\`a di Roma, 
P.le A.~Moro, 5, Roma, Italy. E-mail: \texttt{l.rossi@uniroma1.it}} 
 \ and Oliver Tough\thanks{Department of Mathematical Sciences, Durham University, Upper Mountjoy Campus, Stockton Rd, Durham DH1 3LE. E-mail: \texttt{oliver.kelsey-tough@durham.ac.uk}}}

\begin{document}

\maketitle

\begin{center}
\emph{Dedicated to the memory of Misha Menshikov.}
\end{center}

\begin{abstract}
We consider the existence and uniqueness of positive, bounded stationary solutions of the Fisher-KPP 
equation with Dirichlet boundary conditions. We establish that stationary solutions are unique, and exist if and only if the generalised principal eigenvalue is strictly positive (supercriticality), in dimension at most $6$. This is sharp: we exhibit examples both of non-uniqueness and of existence at criticality in dimension $7$ or higher, including for the classical reaction $f(u)=u-u^2$. This proves a conjecture of Berestycki and Graham in the affirmative for dimension at most $6$, and in the negative for dimension at least $7$. We thereby show that branching Brownian motion cannot survive globally without surviving locally in dimension $6$ or lower, but can do so in dimension $7$ or higher, mirroring the classical transition of Brownian motion from recurrence to transience from dimension $2$ to $3$. We further show that relaxing the strong-KPP condition results in a more general sharp dimension threshold for the uniqueness/non-existence result. These results are valid without any regularity assumptions on the domain boundary. As a corollary, we show that solutions of the evolution reaction-diffusion equation are always convergent in the long run in dimension at most $6$, whilst we exhibit 
non-convergent solutions in dimension larger than $6$.
\end{abstract}

\tableofcontents

\section{Introduction and main results}

The question of the existence and uniqueness of positive, bounded solutions of 
semilinear elliptic equations in divergence form,
under different types of boundary conditions and structural assumptions on the nonlinearity, has received a great deal of attention both classically and contemporaneously. We will answer these questions for the Dirichlet problem 
and KPP-type nonlinearity. We will then apply this to 
the long-time behaviour of the corresponding 
reaction-diffusion equation and branching Brownian motion, see Sections \ref{sec:evolution}, \ref{section:probabilistic results} respectively for the results in this direction.

More precisely, we deal with the Dirichlet problem
\Fi{s-a}
\left\{\begin{array}{ll}
\partial_i(a_{ij}(x)\partial_j u)+f(x,u)=0, & x\in D,\\[5pt]
u(x)\overset{\xi}{=}0, & x\in\partial D.
\end{array}\right.
\Ff 
The domain $D$ will typically be unbounded, with no regularity assumptions on $\partial D$. 
Because of this, the boundary condition is 
understood in a relaxed sense, that is the one of Berestycki-Nirenberg-Varadhan \cite{BNV1994}, by means of a barrier function $\xi$ (whence the notation
$\overset{\xi}{=}$), which we make precise in Section~\ref{section:Barrier}.
Suffice it to say, if $\partial D$ is locally Lipschitz, this reduces to $u$ vanishing continuously along $\partial D$. As a matter of fact, our results are new even when $D$ is the whole space, in which case one 
neglects the second line of \eqref{s-a}.

It has been conjectured by Berestycki and Graham in 
\cite[Conjecture 1.3]{Berestycki2025} that, when $(a_{ij})$ is the identity matrix and $f$ is independent of $x$ and satisfies the strong-KPP condition
(see \eqref{eq:strongKPP f x} below),
then the solution of \eqref{s-a} is unique. We will show that this is true when the dimension is up to an explicit threshold, which depends on the nonlinearity $f$. 
Conversely, we will exhibit multiplicity of solutions 
for the problem~\eqref{s-a} for any dimension above that threshold, including $(a_{ij})$ is the identity and $f(x,u)=u-u^2$.

We will impose the following standing assumption throughout.
\begin{assum}\label{assum:standing}
The domain $ D$ is a non-empty, open, connected subdomain\commentout{\nota{\BS subset? domain=connected for me. What is 
$C^1(\bar  D)$ if $D$ is not regular?}} of~$\Rm^d$. The coefficient $(a_{ij})$ is uniformly elliptic, bounded and $C^1(\bar  D)$. Regarding the reaction term, we assume that $f\in C^1(D\times \Rm_{\geq 0})$, and \commentout{\nota{Increased regularity, attempted to keep simple. Need $\partial_sf$ cts at $0$ just for persistence, don't need cts elsewhere. Need H{\"o}lder regularity of $f$ for Schauder theory. 
{\BS We could write $\|\partial_s f\|_{L^\infty(D\times(0,S'))}$.}
Old assumption: Concerning the regularity of $f$, we assume that $s\mapsto f(x,s)$ is in $C^1([0,\delta))$ for some $\delta>0$ (which doesn't depend upon $x$) and that it is locally Lipschitz continuous, uniformly with respect to $x\in D$. }}
 \Fi{0}f(x,0)=0\quad\text{ for
all}\quad x\in D,\Ff 
\Fi{S}\exists S>0\ \text{ such that}\quad
\sup_{x\in D}f(x,s)< 0\quad \text{ for all  \ $s> S$,}\Ff 
\Fi{fLip}\sup_{\substack{x\in  D,\\ 0\leq s< S'}}\lvert \partial_sf(x,s)\rvert 
 <\infty\quad\text{ for all \ $S'>0$.}\Ff 
\end{assum}

\begin{rmk}\label{rmk:generalising f Holder}
One can replace in all of our results the assumption that $f\in C^1(D\times \Rm_{\geq 0})$ with the more general but less straightforward assumptions that $D\times \Rm_{\geq 0}\ni (x,s)\ra f(x,s)$ is locally H{\"o}lder continuous (meaning that both the H{\"o}lder exponent and constant can depend upon the choice of non-empty, pre-compact neighbourhood), $s\mapsto f(x,s)$ is in $C^1(\Rm_{\geq 0})$ for all $x\in D$, and $\partial_sf\in C^0(D\times \Rm_{\geq 0})$, as these are all that we will use. We will employ the latter generalisation in Section \ref{section:probabilistic proofs}.
\end{rmk}

Any non-negative and bounded solution of \eqref{s-a} is either everywhere $0$, or everywhere positive, by the 
strong maximum principle. 
We focus on the latter case. 
Assumption \eqref{S} implies that bounded solutions are smaller than or equal to $S$, along
with the existence of a 
{\em maximal} solution, see Proposition \ref{prop:maximal solution} below.

Throughout the paper, we denote
\[
Lu:=\partial_{i}(a_{ij}(x)\partial_{j} u).
\]
To examine the question of existence, we introduce the linearised operator about $0$ associated with
\eqref{s-a},
$$\L u:=\partial_{i}(a_{ij}(x)\partial_{j} u)+\partial_s f(x,0)u,$$
so that $\L=L+\partial_sf(x,0)$. Note that these can readily be put into non-divergence form, as $a\in C^1(\bar D)$. 
We then consider the notion of the generalised principal eigenvalue, introduced by 
Berestycki, Nirenberg and Varadhan in \cite{BNV1994}, which, for a given
elliptic operator $\L$, reads \footnote{Opposing sign conventions for $\lambda$ are employed in PDE and probability, so that whereas \eqref{eq:gen prin evalue} is the definition common in PDEs, which we use here, the typical definition employed in probability differs from this by a factor of $-1$.}
\begin{equation}\label{eq:gen prin evalue}
\lambda(-\L):=\sup\{\lambda\in \Rm:\exists \phi\in W^{2,d}_{\text{loc}}( D)\text{ such that $-\L \phi\geq \lambda \phi$ and $\phi>0$ in $D$}\}.
\end{equation}
When $D$ is bounded and smooth, this notion coincides with the classical 
principal eigenvalue of $-\L$ in $D$, under Dirichlet boundary condition.
We will sometimes consider the generalised \pe\ on a subdomain $ D'$ of $D$, 
that we will denote by $\lambda(-\L; D')$.

\commentout{
\begin{rmk}[Generalised principal eigenvalue in PDEs versus probability]
The definition \eqref{eq:gen prin evalue} follows the sign convention adopted in the PDE literature, rather than the convention typically followed in the probabilistic literature. In the probabilistic literature, one instead defines \cite{}
\[
\lambda(\L):=\inf\{\lambda\in \Rm:\exists u\in W^{2,d}_{\text{loc}}( D)\text{ such that $\L u\leq \lambda u$, $u(x)>0$ for all $x\in  D$}\},
\]
which differs from \eqref{eq:gen prin evalue} by a factor of $-1$. The probabilistic literature follows this convention so that positive $\lambda$ corresponds to growth (loosely speaking) and negative $\lambda$ decay.

Since our main results are PDE results, we instead follow the PDE convention, i.e. the definition in \eqref{eq:gen prin evalue}. Therefore for us, $\lambda>0$ corresponds to decay and $\lambda<0$ corresponds to growth.
\end{rmk}}

\subsection{Dimension threshold for unique steady states}
In the following and throughout this article, $\ell>0$ is a parameter. We will specify in our uniqueness results that $\ell\geq 1$, since the condition there cannot be satisfied for $\ell\in (0,1)$. To characterise the existence of positive, bounded solutions to \eqref{s-a}, we will need to consider the following condition:
\commentout{
\begin{equation}\label{eq:pve fs}
\inf_{x\in  D}\partial_sf(x,0)>0,
\end{equation}
and}
\begin{equation}\label{eq:existence condition}
\text{there exists $\epsilon>0$ such that \;$f(x,s)\leq \partial_sf(x,0)s-\epsilon s^{1+\ell}$ \; for $x\in  D$,
$s>0$.}
\end{equation}
This strengthens the weak KPP condition, given here by \eqref{eq:weakKPP f},
and prevents in particular $f$ from being linear.

We prove the following.
\begin{theo}\label{theo:nonexistence at criticality}
Suppose that Assumption~\ref{assum:standing} holds.
Suppose further that $f$ satisfies \eqref{eq:existence condition} for
some $\ell>0$ and that the space dimension satisfies $d\leq 2+\frac{4}{\ell}$. 
Then \eqref{s-a} admits a positive, bounded solution if and only if $\lambda(-\mathcal{L})<0$.
\end{theo}

\begin{rmk}
In this and all of our results, $D\subseteq \Rm^d$ may be replaced with $\lvert D\cap B(0,R)\rvert=O(R^d)$, with the same constraints on $d$ and no changes to the proofs.
\end{rmk}

To examine the question of uniqueness, we will need to strengthen assumption~\eqref{eq:existence condition}
to the following:
\Fi{concav1}
\left\{\begin{array}{l} \text{there exists a constant
$\sigma>0$ such that}\\[5pt]
\displaystyle  f(x,s_2)/s_2-f(x,s_1)/s_1\leq -\sigma (s_2-s_1)^{\ell}\\[5pt]
\text{for all } x\in D, \ \ 0<s_1<s_2.
\end{array}\right.
\Ff


\commentout{
To examine the question of uniqueness, we will need the following condition on the function $s\mapsto f(x,s)$,
\Fi{concav1}\left\{\begin{array}{l} \text{there exists a positive constant
$\sigma>0$ such that}\\
f(x,s_2)\leq f(x,s_1)+\partial_s f(x,s_1)(s_2-s_1)-\sigma (s_2-s_1)^{1+\ell}\\
\text{for all} \quad 0<s_1<s_2\leq S\quad\text{and}\quad x\in D.
\end{array}\right.
\Ff
\begin{rmk}
The condition $f(x,s_2)\leq f(x,s_1)+\partial_s f(x,s_1)(s_2-s_1)-\sigma (s_2-s_1)^{1+\ell}$ for $0<s_1<s_2\leq S$ and $x\in  D$ cannot be satisfied for any $\ell \in (0,1)$, hence we stipulate $\ell\geq 1$.
\end{rmk}
\begin{rmk}
A sufficient condition for \eqref{concav1} is that $s\mapsto f(x,s)$ is $C^2([0,S])$ for $x\in  D$, and that there exists a positive constant $\sigma>0$ such that $\partial_{ss}f(x,s)\leq -\sigma s^{\ell-1}$ for all $0\leq s\leq S$ and $x\in  D$. This can be checked from Taylor's theorem with the integral form of the remainder.
\end{rmk}}
\begin{theo}\label{theo:uniqueness}
Suppose that Assumption~\ref{assum:standing} holds.
Suppose further that $f$ satisfies \eqref{concav1} for
some $\ell>0$ and that the space dimension satisfies $d\leq 2+\frac{4}{\ell}$.
Then \eqref{s-a} admits at most one positive, bounded solution.
\end{theo}

\begin{rmk}
Some comments on the assumptions~\eqref{eq:existence condition},  \eqref{concav1} are in order.
\begin{enumerate}
\item
The condition \eqref{concav1} cannot be satisfied for any $\ell\in(0,1)$, hence we stipulate~$\ell\geq1$. Notice that the larger $\ell$ is, the weaker the condition.
\item
Letting $s_1\searrow0$ in \eqref{concav1}, under the Assumption \ref{assum:standing}, 
one recovers~\eqref{eq:existence condition}.
\item
If $f$ satisfies the assumption \eqref{0} and moreover $s\mapsto f(x,s)$ is of class $C^2$ with
\[\sup_{\Omega\times[0,+\infty)}  \partial_{ss}f<0,\]
then condition \eqref{concav1} holds  with $\ell=1$ (and hence~\eqref{eq:existence condition} too).
This is seen by computing, for  $s>0$,
 \[\partial_s\Big( \frac{f(x,s)}{s}\Big)
=\frac{\partial_s f(x,s)s-f(x,s)}{s^2}=\frac1{s^2}
\int_0^s \partial_{ss}f(x,\tau)\tau d\tau\leq\frac12 \sup\partial_{ss}f.\]
Thus \eqref{concav1} is a weak concavity condition
Under the uniform negativity of $\partial_{ss}f$, the dimensional threshold in 
Theorems \ref{theo:nonexistence at criticality} and \ref{theo:uniqueness} becomes $d\leq6$.
\end{enumerate}
\end{rmk}

Both Theorems \ref{theo:nonexistence at criticality} and \ref{theo:uniqueness} are sharp, as demonstrated by the following.

\begin{theo}\label{thm:c-e}
Suppose that the space dimension satisfies $d> 2+\frac{4}{\ell}$ for some $\ell>0$. Then there exists a positive,
bounded and radially symmetric $c\in C^{\infty}(\Rm^d)$ 
such that $\lambda(-\mathcal{L})=\lambda(-(L+c))=0$, but the equation
\begin{equation}\label{eq:counterexample intro}
\Delta u+c(x)u-u^{1+\ell}=0, \quad x\in\Rm^d,
\end{equation}
has infinitely many positive, bounded and radially symmetric solutions.
\end{theo}

The function $c$ is explicit and it is a Hardy-type potential. 
A more detailed statement is given by Theorem \ref{thm:Hardy_spercritical} below. 
We also show there that multiplicity of solutions can occur even in the case
$\lambda(-\mathcal{L})<0$.

We will extend these counterexamples in Section \ref{section:homogeneous counterexample} by constructing a locally (but not globally) analytic domain on which we have non-uniqueness of positive steady states. We will do this by the probabilistic method: we construct a random domain on which this occurs with positive probability.
\begin{theo}\label{theo:existence nonuniq d geq 7}
For $d\geq 7$ there exists a locally analytic domain $D\subseteq\Rm^d$ on which~\eqref{s-a} has multiple positive, bounded solutions for the reaction $f(x,s)=s-s^2$.
\end{theo}
This provides a counterexample to a conjecture of Berestycki and Graham for $d\geq 7$ \cite{Berestycki2025,Berestycki2026}.

\begin{example}
Consider the case of classical KPP nonlinearities,
\[
\left\{\begin{array}{ll}
\partial_i(a_{ij}(x)\partial_j u)+c(x)u-u^2=0, & x\in D,\\
u(x)=0, & x\in\partial D.
\end{array}\right.
\]
Then we have the existence of a unique positive, bounded solution when $\lambda(-(\Delta+c))<0$, as well
as non-existence when $\lambda(-(\Delta+c))\geq0$, for $d\leq 6$, but counterexamples to both of these for $d\geq 7$.



\end{example}
\begin{example}
Consider the case of Allen-Cahn nonlinearities,
\[
\left\{\begin{array}{ll}
\partial_i(a_{ij}(x)\partial_j u)+c(x)u-u^3=0, & x\in D,\\
u(x)=0, & x\in\partial D.
\end{array}\right.
\]
Then we have uniqueness of positive, bounded solutions, and non-existence when $\lambda(-(\Delta+c))=0$, for $d\leq 4$, but counterexamples to both of these for $d\geq 5$.
\end{example}

\begin{example}\label{ex:dim3}
Consider in dimension $d=3$ the equation
\[
\left\{\begin{array}{ll}
\partial_i(a_{ij}(x)\partial_j u)+c(x)u-u^{1+\ell}=0, & x\in D,\\
u(x)=0, & x\in\partial D.
\end{array}\right.
\]
Then we have uniqueness of solutions, and non-existence when $\lambda(-(\Delta+c))=0$, for $\ell\leq 4$, but counterexamples to both of these for $\ell>4$.
\end{example}

Example \ref{ex:dim3} 
demonstrates that if $f$ is only assumed to be strictly convex, even uniformly in $x\in  D$, 
i.e.~$\sup_{x\in  D}\partial_{ss}f(x,s)<0$ for all $s>0$,
then one can have non-uniqueness and existence at criticality even when $d=3$. 

Conversely, in dimension $d\leq2$, the concavity-type assumption on $f$ can be further relaxed, and reduced to
the standard strong-KPP condition. This is essentially due to Nordmann \cite{Nord}: we obtain what follows by combining his critical maximum principle with some simple arguments in \ref{sec:uniq from non-exis}. More precisely, one needs
%
\begin{equation}\label{eq:mediumKPP f}
(0,+\infty)\ni s\mapsto f(x,s)/s \quad \text{is non-increasing for} \; x\in  D.
\end{equation}
along with the strong-KPP condition
\begin{equation}\label{eq:strongKPP f x}
(0,+\infty)\ni s\mapsto f(x,s)/s \quad \text{is strictly decreasing},
\end{equation}
for $x$ on a set of positive measure.
We will also need the following technical assumption.
\begin{assum}\label{assum: extra assumption for Nordmann}
The coefficient $a^{ij}$ is $C^{1,\alpha}_{\loc}(D)$, the boundary $\partial D$ is $C^{2,\alpha}_{\loc}$, and the reaction $f$ is $C^{0,\alpha}_{\loc}(D\times \Rm_{\geq 0})$, for some H{\"o}lder exponent $\alpha\in (0,1)$.
\end{assum}
The above differ from being locally H{\"o}lder $C^k$ in that the H{\"o}lder exponent $\alpha$ is uniform, even while the H{\"o}lder constant depends on the neighbourhood.
\begin{rmk} \commentout{\nota{\BS I would rather remove $f\in C^1(D)$ from \ref{assum: extra assumption for Nordmann}
and explain this matter after Theorem {theo:existence and uniqueness two dimensions.
We should also say that this is Nordmann's result}.}
}In light of the assumption $f\in C^1(D)$, the above assumption on the reaction term is unnecessary, but we keep it so that Remark \ref{rmk:generalising f Holder} remains valid.
\end{rmk}

\begin{theo}\label{theo:existence and uniqueness two dimensions}
Suppose that Assumptions~\ref{assum:standing} and \ref{assum: extra assumption for Nordmann} hold, 
and that 
$d\leq 2$. Suppose further that $f$ satisfies \eqref{eq:mediumKPP f} and also the
strong-KPP condition~\eqref{eq:strongKPP f x} for all $x$ in some set $E\subseteq  D$ with positive measure.
Then \eqref{s-a} admits a positive, bounded solution if and only if $\lambda(-\mathcal{L})<0$, and this 
is unique when it exists.
\end{theo}

\begin{rmk}
In fact, is clear from the proof that
the existence part of the above theorem holds in greater generality. We say that $f$ is weak-KPP if it satisfies
\begin{equation}\label{eq:weakKPP f}
f(x,s)\leq f_s(x,0)s\quad \text{for  } \; x\in  D,\; s>0.
\end{equation}
Then if we replace the assumptions \eqref{eq:mediumKPP f}-\eqref{eq:strongKPP f x} 
with \eqref{eq:weakKPP f} along with the inequality in~\eqref{eq:weakKPP f} being strict 
for $x$ on a set of positive measure, the existence of a positive and bounded solution is still equivalent to $\lambda(-\mathcal{L})<0$.
\end{rmk}

It is evident that \eqref{eq:mediumKPP f} on its own, together with Assumptions \ref{assum:standing},
is insufficient for uniqueness, even in bounded domains. 
Indeed, for any given elliptic operator $-L$ on a bounded domain $D$, considering
$f(x,s):=\lambda(-L)s-((s-1)_+)^2$ one has that any sufficiently small multiple of the principal 
eigenfunction associated with $\lambda(-L)$ solves $-Lu=f(u)$.


\subsection{Long-time behaviour of reaction-diffusion 
equations}\label{sec:evolution}

The celebrated Fisher-KPP equation
\begin{equation}\label{eq:FKPP 1d}
\partial_t\eta=\Delta \eta+\eta(1-\eta),\quad (t,x)\in \Rm_{>0}\times \Rm,
\end{equation}
was introduced independently and at the same time by Fisher and Kolmogorov, Petrovskii and Piskunov in 1937 as a model for the spatial spread of an advantageous allele in a biological population \cite{KPP1937,Fisher1937}. This has since developed into the theory of reaction-diffusion equations, which in the self-adjoint heterogeneous setting reads:
\begin{equation}\label{eq:general reaction-diffusion}
\left\{\begin{array}{ll}
\partial_t\eta=\partial_i(a_{ij}(x)\partial_j \eta)+f(x,\eta), 
& (t,x)\in \Rm_{>0}\times  D,\\ 
\eta(t,x)= 0, & (t,x)\in\Rm_{>0}\times \partial  D,\\
\eta(0,x)=\eta_0(x),& x\in  D.
\end{array}\right.
\end{equation}
%
%
These model a wide variety phenomena, in fields such as chemistry, ecology, population genetics and sociology. 

\commentout{\begin{Luca}
\BS
For which domain, and in which sense the boundary condition?

Recall: 
$ D$ has locally Lipschitz boundary, by which we mean the following. For all $x\in \partial  D$ there exists $r>0$ (dependent upon $x$) such that we can write $B(x,r)\cap  D$ as
\[
\{(x',y)\in B(0,r):x'\in \Rm^{d-1},y\in \Rm, y>\phi(x')\},
\]
for some Lipschitz function $\phi$ (where $\phi$ and its Lipschitz constant may be dependent upon $x$), after some rotation and translation under which $x$ is mapped to $0$. Using this, we interpret the boundary condition in \eqref{s-a} to be that $u$ vanishes continuously along $\partial  D$.
\end{Luca}}

Here we assume that $ D$ has locally Lipschitz boundary, by which we mean the following. For all $x\in \partial  D$ there exists $r>0$ (dependent upon $x$) such that we can write $B(x,r)\cap  D$ as
\[
\{(x',y)\in B(0,r):x'\in \Rm^{d-1},y\in \Rm, y>\phi(x')\},
\]
for some Lipschitz function $\phi$ (where $\phi$ and its Lipschitz constant may be dependent upon $x$), after some rotation and translation under which $x$ is mapped to $0$. Using this, we interpret the boundary condition in \eqref{s-a} to be that $u$ vanishes continuously along $\partial  D$. 

The initial datum $\eta_0$ in \eqref{eq:general reaction-diffusion}
will always be assumed to be in $L^\infty(D)$, 
and it is understood to be attained in the $L^1_{loc}(\bar D)$ sense.
This guarantees that the solution exists and it is unique  
in the class of functions which are bounded in space, locally in time. 
This is the class of solutions we will always refer to.

The classification of positive, bounded steady states, i.e.~solutions of the elliptic equation \eqref{s-a},
is essential to understanding the long-run behaviour of the parabolic equation \eqref{eq:general reaction-diffusion}. 
It is natural to expect that under the strong Fisher-KPP condition, solutions  
always converge towards a steady state, and this is supported 
by all known results in  the literature, concerning specific cases (periodic setting, cylindrical domains, etc.).
Under the Assumption~\ref{assum:standing} and the 
weak-KPP condition \eqref{eq:weakKPP f}, 
the following are well-known, and can be deduced e.g.~using the proof of Proposition~\ref{prop:maximal solution} and the comparison principle.
\begin{enumerate}[$1.$]
\item If a solution 
of \eqref{eq:general reaction-diffusion} admits a long-time limit,
then the latter must be a steady~state.  
\item If there is no positive, bounded steady state (which is true in particular if ${\lambda(-\mathcal{L})>0}$), 
then $\eta(t,x)\ra 0$ locally uniformly as $t\ra\infty$, for any initial datum~$\eta_0\geq 0$.   
\item If $\lambda(-\mathcal{L})<0$ and the positive, bounded steady state is unique, then this 
is the attractor of \eqref{eq:general reaction-diffusion} for all 
non-negative and non-trivial initial data.
\end{enumerate}

Owing to the above properties, we have the following byproduct of our elliptic results
in not-too-high dimension, Theorems \ref{theo:nonexistence at criticality} and \ref{theo:uniqueness}.

\begin{cor}
Suppose that Assumption~\ref{assum:standing} holds and that $\partial  D$ is locally Lipschitz.
Suppose further that $f$ satisfies \eqref{concav1} 
for some $\ell>0$, and that the space dimension satisfies $d\leq 2+\frac{4}{\ell}$. 
Let $\eta$ be a bounded solution of~\eqref{eq:general reaction-diffusion} 
with $\eta_0\geq0$ not almost everywhere zero. 
Then the following occurs:
\begin{enumerate}[$(i)$]
\item if $\lambda(-\mathcal{L})\geq0$ then $\eta(t,x)$ converges to $0$
as $t\to+\infty$, locally uniformly in $x\in \bar D$;
\item if $\lambda(-\mathcal{L})<0$ then $\eta(t,x)$ converges to 
the unique positive bounded solution of~\eqref{s-a} as $t\to+\infty$, locally uniformly in $x\in \bar D$.
\end{enumerate}
\end{cor}

This result completely describes the long-time behaviour of solutions to~\eqref{eq:general reaction-diffusion} 
when the dimension is smaller than or equal to the threshold: solutions of
are always convergent; whether they converge to  the 
positive steady state (persistence) or to 0 (extinction) 
is solely determined by the sign
of the principal eigenvalue.

Instead, in higher dimension, we show that there are solutions which do not converge
in the long run. 
In fact we construct such solutions as soon as there is existence at criticality
or multiplicity of positive steady states
--~which can occur independently of the sign of~$\lambda(-\mathcal{L})$,
according to Theorems~\ref{thm:c-e} and~\ref{thm:Hardy_spercritical}.
This is established in the following result, that we prove for locally Lipschitz 
domains, and without requiring \eqref{S}, that is, replacing
Assumption~\ref{assum:standing} with the following.
\begin{assum}\label{assum:standing_Lip}
The domain $ D$ is a non-empty, open, connected subdomain of $\Rm^d$
with $\partial D$ locally Lipschitz. 
The coefficient $(a_{ij})$ is uniformly elliptic, bounded and $C^1(\bar  D)$. The reaction term $f$ is in $C^1(D\times \Rm_{\geq 0})$ and satisfies
\[
f(x,0)=0\quad\text\quad\text{for
all}\quad x\in D,
\]
\[
\|\partial_s f\|_{L^\infty(D\times(0,S'))}
 <\infty\quad\text{for all $S'<\infty$.}
 \] 
\end{assum}

\begin{theo}\label{thm:oscillating}
    Suppose that Assumption~\ref{assum:standing_Lip} holds
    and that $f$ satisfies the KPP-type condition
    \eqref{eq:mediumKPP f}, together with the strong-KPP
    condition~\eqref{eq:strongKPP f x} for all $x$ in some set $E\subseteq  D$ with positive measure.
Suppose further that one of the following holds:
\begin{enumerate}[$a)$]
\item the stationary problem \eqref{s-a}
admits more than one positive bounded solution;
\item the stationary problem \eqref{s-a}
admits a positive bounded solution and moreover $\lambda(-\mathcal{L})=0$.
\end{enumerate}
Then there exists a positive bounded solution $\eta$
of~\eqref{eq:general reaction-diffusion} 
such that $\eta(t,x)$ does not converge as $t\to+\infty$, for any $x\in D$.
\end{theo}

The above theorem, combined with Theorem~\ref{thm:c-e},
gives an example of non-convergent solution.
To our knowledge, this had never been observed before for Fisher-KPP equations.
Examples of non-convergent \commentout{(nor even {\em quasi-convergent}) }positive solutions have been obtained
by Pol\'a\v cik and Yanagida in~\cite{PolYan} for the equation $\partial_t\eta=\Delta \eta +\eta^p$ in $\R^3$, for suitable $p>1$.
Those solutions are localised in space, i.e.~converge to $0$ as $|x|\to\infty$.

The solution we construct in the proof of Theorem \ref{thm:oscillating} is also localised in space,
and oscillates in time
between two distinct steady states, namely, its {\em $\omega$-limit set} contains 
the two steady~states.
The construction works 
as soon as the two steady states are (neutrally) stable,
in a sense that will be made precise in Section \ref{sec:uniq from non-exis}.
Then the conclusion holds because positive solutions are always stable in the Fisher-KPP case, as we show in
Section \ref{sec:uniq from non-exis} below.
In the case $b)$, one of the two steady states is the $0$ function,
therefore, in such case, the solution 
is locally close to $0$ (extinction) in some large time intervals, and locally close to 
a positive steady state (persistence) in some others.
From the modelling perspective, this means that even if the observed density $\eta$ is very low
for a long time, in high dimensions the mass may eventually concentrate and rise up by drawing from infinity.

For the proof of Theorem \ref{thm:oscillating} we exploit 
the stability of the steady states to design an initial datum for which the solution
is attracted by each one of them along a suitable sequence of time intervals. 
The key point is that stable steady states
(i.e.~such that~the linearised operator has non-negative principal eigenvalue) are locally attractive for
compactly supported perturbations.
The stability result relies on 
the self-adjoint structure of the operator $L$, and might not
hold in the non-self-adjoint case. 
We actually completely describe the long-time
dynamics of compactly supported perturbations, locally in space, as follows:

\begin{theo}\label{thm:cptperturb}
    Suppose that Assumption~\ref{assum:standing_Lip} holds.
Let $\eta$ be a solution of~\eqref{eq:general reaction-diffusion},
with $\eta_0\geq0$ not almost everywhere zero and
having compact support in $\bar D$.
Then the following occurs:
\begin{enumerate}[$(i)$]
\item under the weak-KPP assumption \eqref{eq:weakKPP f},
with a strict inequality for all $x$ in some set $E\subseteq  D$ with positive measure, 
if $\lambda(-\L)\geq0$ then $\eta(t,x)$ converges to $0$
as $t\to+\infty$, locally uniformly in $x\in \bar D$;
\item 
under the KPP-type condition
    \eqref{eq:mediumKPP f} and the strong-KPP
    condition~\eqref{eq:strongKPP f x} for all $x$ in some set $E\subseteq  D$ with positive measure, if $\lambda(-\mathcal{L})<0$
then $\eta(t,x)$ converges to 
the minimal positive bounded solution of~\eqref{s-a} as $t\to+\infty$, locally uniformly in $x\in \bar D$.
\end{enumerate}
\end{theo}

The definition and the existence of the minimal solution are given in 
Section~\ref{sec:extremal}.
We will deduce from Theorem \ref{thm:cptperturb} $(i)$
that positive bounded steady states are locally attractive with respect 
to compact perturbations, see Corollary \ref{cor:stability}
below.
\commentout{
\bigskip

{\bf Questions:} 
\begin{itemize}
\item
When $\lambda(-\mathcal{L})\geq0$, is there
Global extinction or can there be Global survival
by "escaping to infinity"?

\item
can we prove that if there is a unique stationary solution 
then all solutions are convergent?

\end{itemize}

}

\subsection{Overview of the literature}

To our knowledge, the first Liouville-type result for KPP-type equations was due to Rabinowitz \cite{Rabinowitz1971}, who established uniqueness for strong-KPP $f$ on bounded domains. This was later extended by Berestycki to non-self-adjoint operators \cite{Berestycki1981}. \commentout{Here the boundedness of the domain precludes the development of structure ``at infinity'', changing the phenomenological picture one expects. }On the other hand, if one considers the classical FKPP equation \eqref{eq:FKPP 1d} in a moving frame of reference (speed $c\geq 2$), then both the constant $1$ function and the travelling wave for speed $c$ are positive and bounded stationary solutions, so one cannot similarly remove the self-adjointness assumption on unbounded domains.\label{discussion of bounded domains}\commentout{ In subsection \ref{}, we shall outline an argument that the assumption of self-adjointness can be replaced with a Lyapunov-type criterion \OT{Detail/come back to}.}

Much of the literature has focused on the case of homogeneous equations, so that $f=f(u)$ (and $a=\text{Id}$), possibly with an inhomogeneous domain. In this context, we say a reaction $f$ is \textit{positive} if $f(0)=f(1)=0$, $f_{\lvert_{(0,1)}}>0$ and $f_{\lvert_{(1,\infty)}}<0$, which generalises the notion of KPP reaction in a homogeneous setting. If the reaction term also satisfies \eqref{eq:strongKPP f x} at some (hence all) $x\in  D$, then we say it is a strong-KPP reaction. 

For positive reactions on the whole space $\Rm^d$, the constant $1$ function is the unique positive, bounded steady state because of the ``hair-trigger effect'' due to Aronson and Weinberger \cite{Aronson1978}. Uniqueness has similarly been established for strong-KPP reactions with Neumann boundary conditions by Berestycki, Hamel and Nadirashvili in \cite{BHN2010} under a geometric condition on the domain, and then later by the first author without this geometric condition \cite{Rossi2020}. 


The situation becomes much trickier in the presence of a killing boundary, or with inhomogeneity in the reaction term. The first Liouville results for structured, unbounded domains were due to Berestycki, Caffarelli and Nirenberg, who established uniqueness for strong-KPP reactions on cylinders \cite{BCN1}, and for positive reactions on uniformly Lipschitz epigraphs \cite{Berestycki1997}, as part of a broader, extensive study of the qualitative properties of steady states of reaction-diffusion equations \cite{Berestycki1993,BCN1,Berestycki1997}. These are complemented by the much more recent work of Luo and Lu \cite{Lou2019}, who established the uniqueness of steady states for KPP reactions on cones. Recently, Berestycki and Graham \cite{Berestycki2026} studied the effect of geometry on uniqueness for positive reactions, establishing sufficient geometric conditions for uniqueness and counterexamples to the contrary, using the ``stable-compact'' method they introduce. We note that the positive reactions for which they establish non-uniqueness are not strong-KPP. These works are orthogonal to our own, in that they address the question of how uniqueness depends upon the geometry, whilst we answer the question of how uniqueness depends upon the relationship between the dimension and the imposed KPP assumption. 


There has been much work on periodic media, focusing on the asymptotic propagation speed \cite{Gartner1979,Weinberger2002,BHR2005,BHN2005}. In this context, Berestycki, Hamel and Roques have established uniqueness of the stationary solution for strong-KPP reactions \cite{BHR2005}. Other kinds of inhomogeneity have been considered. Berestycki, Hamel and Nadin have established, in the context of general time-dependent inhomogeneous reactions in the whole space \cite{BHN2008}, the uniqueness of a uniformly-positive entire solution that is analogous to steady states in this context. 

Berestycki, Hamel and the first author \cite{BHRossi} considered inhomogeneous equations on the whole space, establishing spectral sufficient conditions under which the steady state is unique, for strong-KPP type reactions. In recent work \cite{Berestycki2025}, Berestycki and Graham have established uniqueness of the steady state under a general condition extending the aforementioned spectral condition, for homogeneous reactions but with an inhomogeneous Dirichlet boundary being the source of inhomogeneity. They conjectured that this spectral assumption is unnecessary in this and their subsequent paper (\cite[Conjecture 1.4]{Berestycki2025}, \cite[Open Question 3]{Berestycki2026}). The sharp threshold established in this paper shows that, for arbitrary dimension, their assumption is necessary, and that it's only for dimension below the threshold that their assumption can be completely dispensed with. Indeed, even in the homogeneous setting they considered, Theorem \ref{theo:existence nonuniq d geq 7} shows that their conjecture is false in $d\geq 7$ for the classical reaction $f(u)=u-u^2$. In fact, the counterexamples in Section \ref{section:infinite} show that this condition is also not sharp in dimensions above the threshold, see Question \ref{question:sharp characterisation}. 

We note that \cite{Berestycki2025} treats strong-KPP reactions for both Dirichlet and Robin boundary conditions in a unified framework. The latter share much of the same difficulties as the former, with the particular added difficulty that one loses the comparison principle with arguments that involve moving the boundary. Consequentially, the results of \cite{Berestycki2026} are restricted to Dirichlet boundary conditions, with the corresponding questions for Robin boundaries being posed as open problems (see \cite[Section 8.2]{Berestycki2026}). Whilst we restrict our attention to Dirichlet boundary conditions, our theorems and proofs would remain valid for Robin or Neumann (locally $C^1$) boundaries.

One can also ask what kind of boundary regularity is needed. The classical results of Rabinowitz on bounded domains \cite{Rabinowitz1971}, along with the recent work of Berestycki and Graham on unbounded domains \cite{Berestycki2025}, require a uniformly $C^{2,\alpha}$ boundary. This gives an interior ball condition, and hence the Hopf lemma, which is important to their arguments. One would also run into problems formulating the spectral condition in \cite{Berestycki2025} without some appropriate uniform regularity. It is natural to ask to what extent this is simply a technical obstruction. Berestycki and Graham conjectured that steady states are unique on bounded Lipschitz domains, which they proved in \cite[Theorem 7.1]{Berestycki2026} using their stable-compact method. This was later proven for general bounded domains for the classical nonlinearity $u(1-u)$ by the second author and Maillard using probabilistic methods \cite[Theorem 4.19]{Maillard2025}. We extend uniqueness to general strong-KPP reactions on general bounded domains in Section \ref{section:bounded}; this is essentially given by the critical maximum principle in \cite{BNV1994}, combined with some simple arguments in Section \ref{sec:uniq from non-exis}. \commentout{{\BS we should say that this is essentially proved in \cite{BNV1994}. I wouldn't state those as
our Theorems}\OT{I think saying it's essentially proved in \cite{BNV1994} is a bit strong, since Henri wrote this as a conjecture in \cite{Berestycki2025}. I may be wrong, but I think the issue is they didn't have the non-existence $\ra$ uniqueness argument because they were working along the proof strategy of Rabinowitz? Maybe the correct thing to say is something like `We show that, in fact, uniqueness for general strong-KPP reactions on general bounded domains follows as a consequence of the critical maximum principle in \cite{BNV1994}.} }

The question of unbounded domains with nonsmooth boundaries seems more subtle. Even if the boundary is locally smooth, this may naturally be non-uniform, which one might imagine could cause problems ``at infinity''. Our proofs (see Section \ref{thm:Liouv6}) use novel arguments avoiding any terms in our estimates coming from the boundary regularity, hence we are able to prove our results for totally general open domains in Section \ref{section:Liouville}.

Finally, we note that other reactions can exhibit wholly different phenomena. Recently, Berestycki, Graham and Wei \cite{Berestycki2025b} considered bistable and certain field-type nonlinearities on epigraphs. For bistable reactions, they established a transition from there being a unique stable solution when the epigraph is bounded from below, to there being infinitely many solutions once the aperture is greater than $\pi$. For the field-type reactions they consider, the transition is from no solution to infinitely many. The steady states of bistable reactions is a deep and extensive topic, for instance being the subject of the celebrated De Giorgi conjecture \cite{DeGiorgi1979,Savin2009,DelPino2008}. 

\subsection{Dimension threshold for the recurrence of BBM}\label{section:probabilistic results}

Branching processes are a staple of probability theory, whose study goes back to the 19th century work of Bienaym{\'e} \cite{bienayme1845loi} and Galton and Watson \cite{GaltonWatson1875} on the survival of single-type branching processes. Bienaym{\'e} stated without proof that such a (non-trivial) branching process survives with positive probability if and only if the mean offspring is greater than $1$, the first correct proof of which was due to Haldane \cite{Haldane1927}; see Kendall \cite{Kendall1975} for a historical account. Branching processes have since found applications into a wide array of areas, including random graphs, quantum gravity, turbulence, front propagation, disordered systems, stochastic geometry, nuclear fission, time series, and population biology. Whereas Bienaym{\'e}-Galton-Watson processes are well-understood since the foundational work of Kolmogorov, Yaglom and Sevast'yanov \cite{Kolmogorov1938,Yaglom1947,Sevastyanov1951}, the modern theory of branching processes now focuses on spatial branching processes, of which branching Brownian motion (BBM) is the canonical example; see \cite{Arguin2011,Bramson1978,Lalley1987,Berestycki2024,Aidekon2013,Arguin2013} for some landmark results on BBM. 

We consider a BBM $\bar X_t=(X^1_t,\ldots,X^{N_t}_t)$ on a non-empty, connected, open domain $D\subseteq \Rm^d$, where particles evolve as independent Brownian motions (with diffusivity $\sqrt{2}$) and are killed instantaneously upon contact with $\partial D$, i.e. when $X^i_{t-}\in \partial D$. Thus $N_t$ is the number of particles alive at time $t$. We make no assumptions on the boundary $\partial D$. To specify the branching mechanism, we take bounded and locally H{\"o}lder continuous
\[
r(x),\; p_n(x), \quad 0\leq n\leq n_{\max},
\]
for some $n_{\max}<\infty$, such that 
\begin{equation}\label{eq:pn equiv 1}
\sum_{n=0}^{n_{max}}p_n(x)\equiv 1.
\end{equation}
Then independently at Poisson rate $r(x)$, each particle ($X^i_{t}$ say) splits into $n$ particles born at its location, with $n$ chosen independently from the distribution $(p_0,\ldots,p_{n_{\max}})(x)$. The particle that branches is removed, so that with probability $p_0(x)$ we are simply removing it without replacement, and with probability $p_1(x)$ essentially nothing happens. We write $\Pm_x$ for the probability measure under which $\Pm_x(\bar X_0=(x))=1$, for $x\in D$ (i.e. $\bar X_t$ initially consists of a single particle at $x$). 
\begin{defin}\label{defin:nontrivial and unif pve branching}
We say that the branching mechanism is somewhere branching if
\begin{equation}\label{eq:non-trivial branching}
\lvert \{x\in D:\sum_{n=2}^{n_{\max}}p_n(x)>0\}\rvert>0.
\end{equation}
We say that the branching mechanism is \textit{uniformly branching} if
\begin{equation}\label{eq:unif branching}
    \inf_{x\in D}r(x)\sum_{n=2}^{n_{\max}}p_n(x)>0.
\end{equation}
\end{defin}
\begin{rmk}
The boundedness of the offspring distribution is a simplifying assumption, and is not essential. The essential requirement is that the corresponding moment generating function (given in \eqref{eq:nonlinearity f}) satisfies Assumption \ref{assum:standing}.
\end{rmk}

We have the following standing assumption.
\begin{assum}\label{assum:standing prob}
The domain $D\subseteq \Rm^d$ is non-empty, open, and connected. Regarding the branching, $r(x)$ and $(p_n(x),\; 0\leq n\leq n_{\max}$) ($n_{\max}<\infty$), are bounded and locally H{\"o}lder continuous, and satisfy \eqref{eq:pn equiv 1} and \eqref{eq:non-trivial branching}.
\end{assum}

\begin{defin}[Global survival]
We say $\bar X_t$ undergoes global survival if $N_t\geq 1$ for all $t<\infty$. We denote the event of global survival by $\glSurv$.
\end{defin}

\begin{defin}[Local survival]\label{defin:local suvival}
Given a non-empty, open, pre-compact subset $V\subset\subset D$, we say that the branching process $\bar X_t$ undergoes \textit{$V$-local survival} if 
\[
\sup\{t:X^i_t\in D\text{ for some $1\leq i\leq N_t$}\}=+\infty.
\]
We say that $\bar X_t$ undergoes \textit{local survival}, denoted by $\locSurv$, if it undergoes $V$-local survival for every non-empty, open, pre-compact $V\subset\subset D$.
\end{defin}

Then $\locSurv$ is an event, and is equal to $V$-local survival for any non-empty, open, pre-compact $V\subset\subset D$ up to null sets, by a simple Borel-Cantelli argument (see e.g.~\cite[p.20]{Maillard2025}). The events of \textit{global extinction} and \textit{local extinction} are the complements of global and local survival, respectively. We will answer the question of whether global survival without local survival is possible.

The event of global survival is simply that the branching Brownian motion doesn't die out, i.e. that there will always be a particle somewhere. The event of local survival, on the other hand, is the equivalent of recurrence for branching processes - it is the event that an arbitrarily chosen neighbourhood will see a particle visit infinitely often. Clearly $\locSurv\subseteq \glSurv$. The event $\glSurv\setminus \locSurv$ is the event that our branching Brownian motion doesn't die out, but it eventually vacates any given neighbourhood. 

If the branching mechanism is trivial (e.g.~by taking $r\equiv 0$) and $D=\Rm^3$, we obtain three-dimensional Brownian motion. Then whilst we have almost-sure global survival, Brownian motion is transient, so we also have almost-sure local extinction. If we add a branching rate supported on a compact set, or decaying sufficiently quickly to $0$, then it's not hard to obtain essentially the same possibility of global survival without local survival in $d\geq 3$. We will instead be interested in the case of uniformly positive branching, and will instead find a threshold between dimensions $6$ and $7$. 

The question of when branching processes survive locally and globally is fundamental \cite{Menshikov1997,Comets1998,Grigoryan2003,Pemantle2001,Lalley1997,Harris1963,Englander2004,Hering1978b,Maillard2025}. However, the possibility of global survival without local survival is a subtle one. Indeed, criteria for the survival of branching processes have largely focused on settings in which the process can be approximated by the process on a bounded subdomain. Whereas this can successfully illuminate the issue of local survival (e.g.~as in \cite{Englander2004}), it cannot illuminate the event of global survival without local survival, wherein particles ``drift off to infinity'', so that the behaviour ``at infinity'' is crucial.


To understand whether local survival and global survival are possible, it will be necessary to consider the generalised principal eigenvalue of the operator
\begin{equation}\label{eq:L prob}
\L=\Delta+c\quad\text{where}\quad c(x):=r(x)\sum_{n=0}^{n_{\max}}p_n(x)(n-1).
\end{equation}
The generalised principal eigenvalue is defined as in \eqref{eq:gen prin evalue},
\begin{equation}\label{eq:gen prin evalue prob section repetition}
\lambda(-\L):=\sup\{\lambda\in \Rm:\exists u\in W^{2,d}_{\text{loc}}(D)\text{ such that $-\L u\geq \lambda u$ and $u>0$ in $D$}\}.
\end{equation}

Engl{\"a}nder and Kyprianou established that local survival has positive probability if and only if $\lambda(-L)<0$ \cite{Englander2004}\footnote{Recall that opposing sign conventions for $\lambda$ are employed in PDE and probability, so that whereas \eqref{eq:gen prin evalue prob section repetition} is the definition common in PDEs, which we use here, the typical definition employed in probability differs from this by a factor of $-1$.}. Recently, Maillard and the second author \cite{Maillard2025} established the same relationship between global survival and another notion of generalised principal eigenvalue, $\lambda'(-L)$ (introduced by Berestycki and Rossi \cite{BR4}), except that this situation is ambiguous when $\lambda'(-L)=0$. In the present self-adjoint setting the two are equal \cite{BR4}. We obtain the following by putting together results from \cite{Maillard2025,Englander2004,BR4}.
\begin{theo}\label{theo:basic facts survival}
Suppose that Assumption~\ref{assum:standing} holds. Then we have:
\begin{enumerate}
\item Each of $\Pm_x(\locSurv)$, $\Pm_x(\glSurv)$ and $\Pm_x(\glSurv\setminus\locSurv)$ is either strictly positive for all $x\in D$, or $0$ for all $x\in D$. \label{enum:strict pve or 0 everywhere}
\item If $\lambda(-\L)<0$, then $\Pm_x(\locSurv)>0$ and $\Pm_x(\glSurv)>0$ for all $x\in D$.\label{enum:nve evalue surv}
\item If $\lambda(-\L)=0$, then $\Pm_x(\locSurv)=0$.\label{enum:0 evalue no loc surv}
\item If $\lambda(-\L)>0$, then $\Pm_x(\locSurv)=0$ and $\Pm_x(\glSurv)=0$ for all $x\in D$.\label{enum:pve evalue no surv}
\end{enumerate}
\end{theo}
Note that $\glSurv$ and $\locSurv$ both being positive probability events doesn't imply they have the same probability. This leads to two natural questions:
\begin{enumerate}
\item Is $\Pm_x(\glSurv)$ positive or $0$ when $\lambda(-L)=0$?
\item Is $\Pm_x(\glSurv\setminus\locSurv)$ positive or $0$? 
\end{enumerate}
It turns out (see \cite{Maillard2025}) that the first question is equivalent to the question of existence at criticality of stationary solutions of the corresponding reaction-diffusion equation, whilst the second question is equivalent (when $\lambda(-L)<0$) to the uniqueness of stationary solutions. This connection along with our results on stationary solutions of the FKPP equations will allow us to prove the following. 


\begin{theo}\label{theo:BBM no glob surv without loc surv}
We suppose that the branching rate is uniformly positive and $d\leq 6$. Then we have:
\begin{enumerate}
    \item If $\lambda(-\L)=0$, then $\Pm_x(\glSurv)=0$ for all $x\in D$.
    \item $\Pm_x(\glSurv\setminus \locSurv)=0$ for all $x\in D$.
\end{enumerate}
\end{theo}

The latter tells us that, when $d\leq 6$, if branching Brownian motion doesn't die out, then it must visit any given neighbourhood at arbitrarily large times. The following theorem tells us that dimension $6$ is sharp. We consider a branching Brownian motion on $\Rm^d$ specified by a binary branching rate $\br(x)$ and killing rate $\kappa(x)$, so that
\[
r(x):=\br(x)+\kappa(x),\quad p_0(x):=\frac{\kappa(x)}{r(x)},\quad p_2(x):=\frac{\br(x)}{r(x)},\quad p_n(x)\equiv 0\text{ for $n\neq 0,2$.}
\]
Then $c=\br-\kappa$, so that $\L=\Delta+\br-\kappa$. The branching mechanism being uniformly positive is the statement that $\inf_{x\in \Rm^d}\br(x)>0$, which will be trivial in the following.

We let $V\in C^{\infty}(\Rm^d)\cap L^\infty(\R^d)$ be an asymptotic Hardy potential, that is,
\Fi{Hardy_potential prob}
\exists\e>0,\qquad
V(x) = \frac{c}{|x|^2} + o\left(|x|^{-2-\e}\right) \quad \text{ as }\; |x| \to \infty.
\Ff
We do not assume $V(x)$ to be radially symmetric, but only to have an asymptotic radial symmetry. We will assume that $c$ satisfies
\Fi{cbounds prob}
	2d-8=H_d-\Big(\sqrt{H_d}-2\Big)^2<c<H_d,
	\Ff
where $H_d$ is the critical Hardy constant
\[H_d:=\left(\frac{d-2}{2}\right)^2.\]

\begin{theo}\label{theo:counterexample prob}
Suppose that $d\geq 7$. We suppose that $\br\equiv 1$, and $\kappa=\max(1-V,0)$, where $V$ is given by \eqref{Hardy_potential prob} for $c$ satisfying \eqref{cbounds prob}. 

Then $\Pm_x(\glSurv\setminus \locSurv)>0$ for all $x\in \Rm^d$, and 
\begin{equation}
\Pm_x(\glSurv)\sim \Theta(\lvert x\rvert^{-1}).
\end{equation}
Moreover both $(1)$ $\Pm_x(\locSurv)\equiv 0$ and $(2)$ $\Pm_x(\locSurv)>0$ for all $x\in \Rm^d$ are possible. We always have possibility (1) when $V\leq H_d\lvert x\rvert^{-2}$ everywhere. On the other hand, when we have $(2)$, then
\begin{equation}
\Pm_x(\locSurv)\sim \Theta(\lvert x\rvert^{-\sqrt{H_d}-\sqrt{H_d-c}}).
\end{equation}

However, if instead of \eqref{cbounds prob} we assume that $0<c< 2d-8$, then
\begin{equation}
    \Pm_x(\glSurv\setminus \locSurv)\equiv 0.
\end{equation}

On the other hand, if $c>H_d$, then
\[
\Pm_x(\locSurv)>0\quad\text{and}\quad \Pm_x(\glSurv\setminus \locSurv)\equiv 0.
\]

\end{theo}
\begin{rmk}[Double non-trivial phase transition]\label{rmk:double threshold}
Taking $V=\min(c\lvert x\rvert^{-2},1)$, we observe a \textit{double non-trivial phase transition} in $c$: neither local survival nor global survival are possible for $c<2d-8$, local survival is impossible but global survival is possible for $c\in (2d-8,H_d)$, and global survival and local survival are both possible and the same event for $c>H_d$.
\end{rmk}
Whilst these theorems are entirely probabilistic, their proof is via PDE arguments. Since there is a dimension threshold, any proof has to incorporate the dimension, and yet in the authors' point of view, there are fewer tools incorporating dimension in probability versus PDE theory. Therefore one might expect that any probabilistic proof should lead to interesting new tools, motivating the following.
\begin{prob}
Prove Theorems \ref{theo:BBM no glob surv without loc surv} and \ref{theo:counterexample prob} using probabilistic arguments.
\end{prob}
The description of the location of the tip, the extremal point process, and their genealogy have been a central object of study in the investigation of branching Brownian motion \cite{Berestycki2024,Arguin2011,Aidekon2013,Arguin2013,Bramson1978,Lalley1987}. If we consider the branching Brownian motion in Theorem \ref{theo:counterexample prob} with $V(x)=c\lvert x\rvert^{-2}$, for $c$ satisfying \eqref{cbounds prob}\commentout{\nota{should we allow $c=H_d$?}}, we obtain a BBM which survives globally with positive probability, but does not survive locally. Consequentially, on the event of global survival, we must have $\inf_i\lvert X^i_t\rvert\ra \infty$ as $t\ra\infty$. Thus we obtain not only an `outer' tip corresponding to the furthest particle from $0$, but also the new possibility of an `inner tip' corresponding to the particle closest to $0$.
\begin{prob}
Describe the location, the extremal point process, and the genealogy for the inner tip.
\end{prob}

The probability of local and global survival correspond to the minimal and maximal steady states of the FKPP equation in Theorem \ref{thm:Hardy_spercritical}. However Theorem \ref{thm:Hardy_spercritical} provides another family of stationary solutions, $(u_\kappa)_{\kappa>0}$, with asymptotics $u_\kappa\sim \kappa \lvert x\rvert^{-\sqrt{H_d}+\sqrt{H_d-c}}$. 
\begin{question}
Is there a natural probabilistic interpretation for $(u_\kappa)_{\kappa>0}$?
\end{question}
In fact, global survival without local survival is also possible on an inhomogeneous domain $D$, with no killing on the interior.
\begin{theo}\label{theo:gl surv loc surv inhom D}
Suppose that $d\geq 7$. Then there exists locally analytic $D\subseteq \Rm^d$ such that $\Pm_x(\glSurv\setminus\locSurv)>0$ for a simple binary branching Brownian motion killed instantaneously at $\partial D$.
\end{theo}

\subsection*{Overview of the paper}

In Section \ref{section:prelim} we will establish a number of preliminary results, which will be essential for proving our main theorem. We highlight here Section \ref{sec:uniq from non-exis}, where we show how to obtain uniqueness of stationary solutions from non-existence at criticality (for a different reaction term). Section \ref{section:Liouville} is dedicated to the proofs of our main theorems on uniqueness of steady states and non-existence at criticality. This is followed by Section \ref{section:infinite}, where we provide our counterexamples above the dimension threshold. Finally in Section \ref{section:probabilistic results}, we will prove our results on branching Brownian motion using the PDE results proven in the previous sections.

\section{Preliminaries}\label{section:prelim}

\commentout{In this section we will establish a number of preliminary results, for proving our main theorems. We will accomplish the following:\OT{Rewrite}
\begin{enumerate}
\item Whereas we state Theorems \ref{theo:nonexistence at criticality} and \ref{theo:uniqueness} for locally Lipschitz boundaries, we will in fact prove generalisation of them without any boundary assumptions. This enables us to state Theorem \ref{theo:BBM no glob surv without loc surv} without any boundary assumptions. We will provide our formulation of the notion of Dirichlet boundary conditions using the barrier functions introduced in \cite{BNV1994}, with a minor adaptation to account for the domain being unbounded.
\item We will then provide the machinery allowing one to obtain uniqueness of steady states from non-existence at criticality, for an appropriately adapted reaction term. This is crucial, as our proof of Theorem \ref{theo:uniqueness} will proceed by proving non-existence at criticality (Theorem \ref{theo:nonexistence at criticality}) and then applying this machinery.
\item In our proof of Theorem \ref{theo:nonexistence at criticality}, it will be necessary to assume the domain is unbounded. We provide here a proof of the case of bounded domains for general boundary conditions. 
\end{enumerate}}

We assume that $ D$ and $a$ satisfy the conditions in Assumption \ref{assum:standing}, which we recall are that $ D$ is a non-empty, open, connected subdomain of $\Rm^d$, and $(a_{ij})$ is uniformly elliptic, bounded and $C^1(\bar D)$. We do not impose any regularity on $\partial D$. We recall that $Lu:=\partial_i(a_{ij}\partial_ju)$. We now explain the meaning of $u\overset{\xi}{=}0$.


\subsection{Dirichlet boundary conditions on general domains}\label{section:Barrier}


In \cite[Section 3]{BNV1994}, Berestycki, Nirenberg and Varadhan constructed on bounded domains a solution, $\varrho$, of
\[
L \varrho=-1,\quad x\in D,
\]
which is positive in $D$ and 
vanishes on $\partial D$ in some sense. They used this as a barrier function, saying that $u\overset{\varrho}{=}0$ on $\partial D$ if $\varrho(x_n)\ra 0$ implies $u(x)\ra 0$ as $n\ra\infty$. 

In fact, $\varrho$ has the following interpretation. Let $(X_t:0\leq t<\tau_D)$ be the diffusion with generator $L$, killed at the time $\tau_D:=\inf\{t>0:X_t\notin D\}$, with $\Pm_x(B_0=x)=1$. Then 
\begin{equation}\label{eq:u0 is exit time}
    \varrho(x)=\expE_x[\tau_D].
\end{equation}
To see this, we recall that $\varrho$ is constructed by taking an ascending sequence $D_j$ of smooth sub-domains of $D$ with union $D$, and defining $u_j$ to be the unique solution of $L u=-1$ on $D_j$ with Dirichlet boundary conditions defined in the usual way. They define $\varrho$ to be the pointwise non-decreasing limit of the $u_j$s. Note that $u_j(x)$ has the interpretation of being the expectation of the exit time started from $x$ of $D_j$, so \eqref{eq:u0 is exit time} follows from the monotone convergence theorem.

If we try to implement the above on an unbounded domain, however, it is no longer clear that $\tau_D$ is $\Pm_x$-integrable. To circumvent this problem, we employ the same procedure to see that 
\begin{equation}
\xi=\expE_x[1-e^{-\tau}]
\end{equation}
is the minimal non-negative solution to 
\begin{equation}
(L-1)\xi=-1,\quad x\in D.
\end{equation}
\begin{defin}
We say that $u\overset{\xi}{=} 0$ on $\partial D$ if, for any bounded sequence $(x_n)$ in $D$, $\xi(x_n)\ra 0$ implies $u(x_n)\ra 0$.
\end{defin}
Whereas $\overset{\xi}{=}$ avoids the above issues issues with $\overset{\varrho}{=}$ on unbounded domains and they are equivalent on bounded domains, we will still need $\varrho$ later in the proof of Proposition \ref{prop:xi vanishes a.s.}.

\commentout{
\begin{Luca}
\BS As far as I understand, the differences between the conditions $\overset{\xi}{=}$
and $\overset{=}{\rho}$ are that $\xi$ and $\rho$ are slightly different,
and the sequence for $\xi$ is bounded.
Do we need both?
This should be clarified.
\end{Luca}
}

\begin{rmk}[$\xi$ as a barrier function]\label{rmk:xi barrier function}
It is easy to see that $u\overset{\xi}{=} 0$ on $\partial D$ and $u\overset{\varrho}{=} 0$ on $\partial D$ are equivalent for bounded domains. On the other hand, if the domain is unbounded, let $\varrho_R$ be the barrier function on $D\cap B(0,R)$ constructed as in \cite{BNV1994}, for $R<\infty$. It is immediate that $u\overset{\xi}{=} 0$ is equivalent to $u(x_n)\ra 0$ for any sequence $(x_n)$ in $ D$ such that, for some $R<\infty$, $\sup_n\lvert x_n\rvert<R$ and $\varrho_R(x_n)\ra 0$. For any bounded or unbounded $D$, $\xi(x_n)\ra 0$ as $n\ra\infty$ for a sequence $(x_n)$ is equivalent to $\tau_{D,n}\ra 0$ in probability as $n\ra\infty$, where $\tau_{D,n}$ is the exit time started from $x_n$.
\end{rmk}

The following proposition is needed for connecting PDEs with Dirichlet boundary conditions with branching processes on general domains with instantaneous killing at the boundary, and is needed to follow the proof of \cite[Proposition A.5]{Maillard2025} (relying on \cite[Theorem A.2]{Maillard2025}) in order to prove Proposition \ref{prop:parabolic equation well-posedness} below.
\begin{prop}\label{prop:xi vanishes a.s.}
We assume Assumption \ref{assum:standing}. Then $\lim_{t\uparrow \tau_D}\xi(X_t)=0$ almost surely, on $\{\tau_D<t\}$. 
\end{prop}
\begin{proof}[Proof of Proposition \ref{prop:xi vanishes a.s.}]
We firstly consider the case that $D$ is bounded. We then observe, by direct checking using conditional expectation rules, that
\begin{equation}\label{eq:mart M}
M_t:=\varrho(X_t)\Ind(\tau_D>t)-\varrho(X_0)+(t\wedge \tau_D)
\end{equation}
is a martingale. Then $\varrho(X_t)$ is a bounded supermartingale, hence Doob's regularity theorem tells us that is has a c\`adl\`ag modification, which we call $\rho_t$. Since $\varrho(X_t)$ is continuous on $t\neq \tau_D$, $\rho_t=\varrho(X_t)$ for $t\neq \tau_D$. Therefore, by Ito's lemma and a localisation argument,
\begin{equation}\label{eq:mart tilde M}
\tilde{M}_t=\rho_{(t\wedge \tau_D)-}-\int_0^{t\wedge \tau_D}L\varrho(X_s)ds-\rho_0
\end{equation}
is a martingale. Taking the difference of \eqref{eq:mart M} and \eqref{eq:mart tilde M}, we see that
\[
\varrho(X_t)\Ind(\tau_D>t)-\rho_{(t\wedge \tau_D)-}
\]
is a martingale. Here we are using that $\int_0^{t\wedge \tau_D}L\varrho(X_s)ds=-(t\wedge \tau_D)$. Since it is a bounded martingale starting at $0$, it has an almost-sure limit as $t\ra\infty$ with expectation $0$, by Doob's $L^1$-convergence theorem. But $\varrho(X_t)\Ind(\tau_D>t)\ra 0$ almost surely as $\ra\infty$, hence $\expE[\lim_{t\ra\infty}\rho_{(t\wedge \tau_D)-}]=\rho_{\tau_D-}=0$. This implies the desired claim on bounded domains.

For unbounded domains, we combine the above with Remark \ref{rmk:xi barrier function}.
\end{proof}

Whereas we won't work directly with parabolic equations, we will use results relying on them. In each of these cases, the arguments we rely on will apply with no changes when we replace the usual notion of Dirichlet boundary conditions on locally Lipschitz domains (i.e. vanishing continuously along $\partial  D\times \Rm_{>0}$) with the following notion for general domains.
\begin{defin}\label{defin:parabolic equations general domains}
Given a solution $u(t,x)$ of a parabolic equation (either linear or of FKPP type), we say that $u\overset{\xi}{=}0$ if $u(t_n,x_n)\ra 0$ for any bounded sequence $(t_n,x_n)$ in $\Rm_{>0}\times  D$ with $t_n\ra t\in \Rm_{>0}$ and $\varrho(x_n)\ra 0$. 
\end{defin}

We consider the parabolic equation,
\begin{equation}\label{eq:parabolic equation}
\begin{split}
&\partial_tu=Lu+f(x,u),\quad x\in D,\; t>0,\\
&u\overset{\xi}{=}0 \text{ on $D\times \Rm_{>0}$},\quad u(t,x)\ra u_0(x)\text{ as $t\ra 0$ in $L^1_{\loc}$.}
\end{split}
\end{equation}
\begin{prop}\label{prop:parabolic equation well-posedness}
Assume Assumption \ref{assum:standing}. There is a unique solution to \eqref{eq:parabolic equation} for all bounded $u_0\geq 0$.
\end{prop}
The proof of this can be obtained by following the proof of \cite[Proposition A.5]{Maillard2025}.

\subsection{Extremal solutions}\label{sec:extremal}

Here we derive some standard facts about {\em extremal} solutions
to \eqref{s-a}. A positive, bounded
solution $v$ is said to be {\em maximal} (resp.~{\em minimal}) if $v\geq u$
(resp.~$v\leq u$) for any positive, bounded solution $u$.

\begin{prop}\label{prop:maximal solution}
Under the Assumption \ref{assum:standing}, the following holds:
\begin{enumerate}[$(i)$]
\item If \eqref{s-a} admits a positive, bounded solution, then it also admits
a maximal positive, bounded solution $v$; moreover 
$v\leq S$, with $S$ from Assumption~\ref{assum:standing}.
\item If $\lambda(-\L)<0$ then \eqref{s-a} admits a positive, bounded minimal solution. 
\end{enumerate}
\end{prop}


\begin{proof}
$(i)$ Consider the evolution problem \eqref{eq:general reaction-diffusion}.
Let $S$ be from condition \eqref{S}. 
By comparison with the ODE $\dot y=\sup_{x\in D}f(x,y)$, we 
see that any solution to \eqref{eq:general reaction-diffusion}
with initial datum $C>S$ 
will eventually be at most $S+\epsilon$ everywhere, for any $\epsilon>0$.
Therefore any bounded solution to~\eqref{s-a}, i.e.~stationary solution
to~\eqref{eq:general reaction-diffusion}, is at most $S$. 
Consider now the solution with initial datum $\eta_0\equiv S$.
Therefore, by the parabolic comparison principle, 
$\eta(t,\.)$ is larger than or equal to any bounded solution to~\eqref{s-a},
for any $t\geq0$.
It also follows from the strong comparison principle that 
$\eta(t,x)$ is non-increasing in $t$, hence
$\overline u:=\lim_{t\ra\infty}\eta(t,x)$ exists (and it is non-negative).
By parabolic estimates, $\overline u$ is a stationary solution, hence 
it is the maximal one.
By the elliptic strong maximum principle, it
is either positive or identically equal to $0$. 
%
%
%
%
%
Thus if \eqref{s-a} admits a positive and bounded solution, then 
$\overline u>0$.


$(ii)$ 
Using Sard's theorem, we take a  sequence of $C^{\infty}$ domains 
$(D_n)_{n\in\N}$ satisfying
\[D_n\subset D_{n+1}\subset\subset D \ \text{ for all $n\in\N$},
\qquad \bigcup_{n\in\N}D_n=D.\]
These can be taken to always be connected by reducing to the connected component containing some fixed $x_{\ast}\in D$.\commentout{\nota{\BS explain that $D_n$ can be taken connected}} We further consider the sequence of Dirichlet
principal eigenfunctions $\varphi_n$ for $-\L$ on $D_n$, i.e.,
positive solutions of
\[
\left\{
\begin{array}{rl}
-\L\varphi_n=\lambda(-\L;D_n)\vp_n
 & \text{in }D_n,\\[5pt]
\vp_n=0 & \text{on }\partial D_n.
\end{array}\right.
\]
It is well known that the sequence 
$(\lambda(-\L;D_n))_{n\in\N}$ is non-increasing, and that 
it converges to $\lambda(-\L)$, see \cite{BR4}.
Hence, by hypothesis, $\lambda(-\L;D_n)<0$ for $n$ large enough.
Taken such $n$, we get, for $\e>0$,
\[-L(\e\vp_n)=(f_s(x,0) +\lambda(-\L;D_n))\e\vp_n\qquad\text{in }\; D_n,\]
which is less than $f(x,\e\vp_n)$ for $\e$ smaller than or equal to some $\e_0$,
depending on the $C^1$ norm of $s\mapsto f(x,s)$.
By \eqref{S} we have $\e_0\max\vp_n\leq S$.
Take $\e\in(0,\e_0]$. The function $\e\vp_n$, extended by $0$ outside $D_n$,
is a generalised subsolution to \eqref{s-a} in $D$, hence if we take
it as an initial datum for~\eqref{eq:general reaction-diffusion},
we infer from the comparison principle that the emerging solution
is non-decreasing in time, thus it converges as $t\to+\infty$ towards
a stationary solution $0<\ul u\leq S$.
We claim that this is the minimal positive solution.
Let indeed~$u$ be a positive and bounded solution to \eqref{s-a} in $D$.
Then letting 
\[k:=\max_{\bar D_n}\frac{\vp_n}u,\] we have that
$\vp_n/k$ touches $u$ from below somewhere in the domain $D_n$.
But then, if $k\geq1/\e_0$ we would have that a positive subsolution 
touches a solution from below, and thus they would coincide by the strong maximum
principle, but this cannot happen because $\vp$ vanishes on $\partial D_n$ while
$u$ does not. This shows that $k<1/\e_0$, that is, $u>\e_0\vp_n$ on $\bar D_n$, 
whence, again by comparison, $u\geq \ul u$ in $D$.
\end{proof}

\subsection{Uniqueness from non-existence at criticality}\label{sec:uniq from non-exis}
In this subsection we show how to obtain the uniqueness of steady states from non-existence at criticality, for an adjusted reaction term, under the 
Assumption \ref{assum:standing} and the KPP-type condition~\eqref{eq:mediumKPP f}.

Suppose that \eqref{s-a} admits a positive, bounded solution $u$,
then it also admits a maximal one, owing to Proposition \ref{prop:maximal solution}, denoted by $\ol u$. We consider the difference $w:=\ol u-u$, and we now
show that $w$ is a subsolution of a stationary FKPP equation whose linearised operator has non-negative generalised principal eigenvalue.
We have indeed that $w$ is a bounded and non-negative solution of
\begin{equation}\label{eq:w-h}
Lw+h(x,w)=0,\quad x\in  D,\qquad 
w\overset{\xi}{=} 0,\quad x\in  \partial D,
\end{equation}
with the function $h$ defined by
\Fi{h}
h(x,s):=f(x,u(x)+s)-f(x,u(x)),\quad x\in  D,\; s\geq0.
\Ff

\begin{prop}\label{pro:nonex-uniqueness}
Suppose that Assumption \ref{assum:standing} holds and 
that $f$ satisfies \eqref{eq:mediumKPP f}.
Suppose further that 
$u$ is a positive, bounded 
solution of \eqref{s-a}.
Then the function $h$ defined by \eqref{h}
satisfies the conditions on the reaction term $f$
in Assumption \ref{assum:standing},
and moreover
\[h(x,s)\leq \frac{f(x,u(x))}{u(x)}\,s
\quad\text{for }\; x\in  D,\; s>0,\]
and the above inequality is strict at all $x$ where $f$ satisfies the strong-KPP condition~\eqref{eq:strongKPP f x}.

Finally, if $f$ satisfies \eqref{concav1} then $h$ satisfies
\[
h(x,s)\leq \frac{f(x,u(x))}{u(x)}\,s-\sigma s^{1+\ell}
\quad\text{for }\; x\in  D,\; s>0.
\]
\end{prop}

\begin{proof}
The regularity of $h$ and property \eqref{fLip} 
are inherited by the corresponding ones of $f$, because $u\in C^2(D)$
by elliptic estimates. We turn to the other conditions.
Property \eqref{0} is straightforward.
In order to check \eqref{S}, we abbreviate $g(x,s):=f(x,s)/s$ and rewrite $h$ as
\[h(x,s)= g(x,u(x)+s)s+u(x)\big(g(x,u(x)+s)-g(x,u(x))\big).\] 
The function $g(x,s)$ is non-increasing in $s\in(0,+\infty)$ 
by \eqref{eq:mediumKPP f},
and it is negative for $s> S$ by \eqref{S}.
As a consequence, $h$ fulfils \eqref{S}.

To show the remaining properties, we now write
\[h(x,s)= g(x,u(x))s+\big(u(x)+s\big)\big(g(x,u(x)+s)-g(x,u(x))\big),\]
whence $h(x,s)\leq g(x,u(x))s$ for $x\in D$, $s>0$,
and the inequality is strict whenever $f$ satisfies~\eqref{eq:strongKPP f x}.
Moreover, if $f$ satisfies \eqref{concav1} then
\[
h(x,s)\leq g(x,u(x))s-(u(x)+s)\sigma s^\ell
<g(x,u(x))s-\sigma s^{1+\ell},
\quad\; x\in  D,\; s>0.
\]
The proof is concluded.
\end{proof}

We have shown that the difference $\ol u-u$ is a bounded and non-negative 
subsolution for the operator
\[
\L_u w:=Lw+\frac{f(x,u(x))}{u(x)}\,w,
\]
and clearly $\lambda(-\L_u)\geq 0$, as it is seen by taking $u$ in its definition.

\subsection{Bounded domains}\label{section:bounded}
Whereas we are concerned with self-adjoint $L$ under the standing assumptions, Assumption \ref{assum:standing}, the following results on bounded domains remain true with no changes to the proofs if we replace $L$ with the non-divergence form operator $a_{ij}(x)\partial_{ij}+b_i(x)\partial_i$, for $a$ uniformly elliptic, continuous and bounded, and $b$ bounded and measurable. Thus \cite[Conjecture 1.3]{Berestycki2025} is true for general bounded, open sets.

We will impose the following assumption.
\begin{assum}\label{assum:generalisation of weak kpp for non-existence for uniq bded}
There exists bounded and measurable $c:D\ra \Rm$ such that
\begin{equation}\label{eq:inequalities for general weak KPP bded}
\lambda(-(L+c))\geq 0\quad\text{and}\quad f(x,s)\leq sc(x)\quad\text{for}\quad x\in D,\; s\geq 0.
\end{equation}
Moreover there exists $E\subset D$ with positive measure
such that the latter inequality in \eqref{eq:inequalities for general weak KPP bded} is strict for $x\in E$.
\end{assum}
\begin{rmk}\label{rmk:generalisation to c}
Assumption \ref{assum:generalisation of weak kpp for non-existence for uniq bded} in the case $c(x)=f_s(x,0)$ recovers 
the weak-KPP condition \eqref{eq:weakKPP f},
with the strict inequality for $x\in E$, along with $\lambda(-\L)\geq 0$. We 
extend here those assumptions because, for the uniqueness result,
we will consider an adjusted reaction term, as done in 
Section \ref{sec:uniq from non-exis}, for which Assumption 
\ref{assum:generalisation of weak kpp for non-existence for uniq bded} 
holds with $c$ different from the linearisation at $0$. A similar generalisation will be needed for our main theorem on unbounded domains, Theorem \ref{thm:Liouv6}.
\end{rmk}
\commentout{

\begin{equation}\label{eq:weakKPP f prelim sec}
f(x,s)\leq f_s(x,0)s,\quad \text{for} \; x\in  D,\; s\in (0,S].
\end{equation}
In place of the weak-KPP $\lambda(-(L+f_s(x,0)))\geq 0$ a

\OT{Replace $f_s$ with $c$}

We prove the following on non-existence at criticality.}
\begin{theo}\label{theo:nonexistence at criticality bded domains}
Suppose that Assumptions \ref{assum:standing} and 
\ref{assum:generalisation of weak kpp for non-existence for uniq bded}
hold, and that $D$ is bounded. Then \eqref{s-a} does not admit a positive, bounded solution.
\end{theo}
\begin{proof}
We suppose that $u$ is a positive and bounded solution of \eqref{s-a}. Then \eqref{eq:weakKPP f} yields
\[
(L+c(x))u+[f(x,u)-c(x)u]\leq 0.
\]
Since $a\in C^1(\bar  D)$ and $ D$ is bounded, $L$ can be put into non-divergence form with coefficients satisfying the assumptions of \cite{BNV1994}, whence we can apply \cite[Corollary 2.2]{BNV1994}. The latter tells us that either $\lambda(-(L+c(x)))<0$ (which is not the case here), or $\lambda(-(L+c(x)))=0$ and $u$ is a multiple of the 
generalised principal eigenfunction $\phi$ for $-(L+c(x))$ (which exists by \cite{BNV1994}). 

We now suppose that we have the latter possibility, so that $u=k\phi$ for some $k\in \Rm_{>0}$. Then $(L+c(x))u=(L+f_s(x,0))(k\phi)\equiv 0$, hence $f(x,u)-c(x)u\equiv 0$. However, this contradicts that the inequality in \eqref{eq:weakKPP f} is strict for $(x,s)\in E\times (0,S]$.
\end{proof}


The uniqueness result on bounded domains immediately follows.

\begin{theo}\label{theo:uniqueness bded}
Suppose that Assumption \ref{assum:standing} holds and that 
$ D\subseteq \Rm^d$ is bounded. 
Suppose further that $f$ satisfies \eqref{eq:mediumKPP f} and also the
strong-KPP condition~\eqref{eq:strongKPP f x} for all~$x$ in some set $E\subseteq  D$ with positive measure.
Then \eqref{s-a} admits at most one positive, bounded solution.
\end{theo}

\begin{proof}
Suppose, for contradiction, that there exist two distinct solutions of \eqref{s-a}. Then by Proposition \ref{prop:maximal solution} there exist
$0<u< \ol u\leq S$ solutions of \eqref{s-a}, with $\ol u$ being the maximal solution. 
The difference $w:=\ol u-u$ satisfies \eqref{eq:w-h} with $h$ 
defined by \eqref{h}. 
Proposition \ref{pro:nonex-uniqueness} implies that 
the function $h$ fulfils the conditions 
in Assumption \ref{assum:standing}, and also in Assumption
\ref{assum:generalisation of weak kpp for non-existence for uniq bded},
with $c(x)=f(x,u(x))/u(x)$ (notice 
that $\lambda(-L+c)\geq 0$ just by taking $u$ in its definition).
Theorem \ref{theo:nonexistence at criticality bded domains} then gives a contradiction.
%
%
\end{proof}

\subsection{Proof of Theorem \ref{theo:existence and uniqueness two dimensions}}

Theorems \ref{theo:nonexistence at criticality bded domains} and \ref{theo:uniqueness bded} remain true if we replace the assumption that $ D\subseteq \Rm^d$ is bounded with the assumption that $d\leq 2$, and Assumption \ref{assum: extra assumption for Nordmann}. We note that, as discussed on Page \pageref{discussion of bounded domains}, self-adjointness then becomes necessary.

The key ingredient is the following critical maximum principle by S. Nordmann.
\begin{theo}[Critical maximum principle, Theorem $2$ of \cite{Nord}]\label{thm:Nord}
Suppose that $d\leq 2$, that $ D$ is non-empty, open and connected, $\partial D$ is $C^{2,\alpha}_{\loc}$ smooth, that $a$ is bounded, uniformly elliptic and $C^{1,\alpha}_{\loc}$, and $c$ is $C^{0,\alpha}_{\loc}$, for some $\alpha\in (0,1)$. Suppose further that $\lambda(-(L+c))\geq 0$. Then if $u$ is bounded, non-negative and satisfies $(L+c)u\geq 0$, then either (1) $u\equiv 0$, or (2) $u$ satisfies
\Fi{s-a_linear}
\begin{cases}
(L+c)u=0 & \text{in } D,\\
u=0 & \text{on }\partial D,
\end{cases}
\Ff
and $\lambda(-(L+c))=0$.
\end{theo}

This result is derived in \cite{Nord} for more general boundary conditions. It relies on a 
Liouville-type result by H.~Berestycki, L.~Caffarelli, and L.~Nirenberg \cite{BCN1},
from which the constraint on the dimension comes from. It fails in dimension $d\geq3$. Note that \cite{Nord} is unpublished as the author has left academia, however we have carefully checked the proof of \cite[Theorem 2]{Nord}. 


We use Theorem \ref{thm:Nord} in place of Theorem \ref{theo:uniqueness bded}, with Assumption \ref{assum: extra assumption for Nordmann} ensuring that the $C^{k,\alpha}_{\loc}$ assumptions in Theorem \ref{thm:Nord} are satisfied. Note that we are using here that the H{\"o}lder exponent in De Giorgi estimates depends only on the dimension and ellipticity constants, not the distance to the boundary, to ensure that the hypothesised solution $u$ is $C^{0,\alpha}_{\loc}$. Thus we have Theorem \ref{theo:existence and uniqueness two dimensions}.\qed

\section{Liouville results}\label{section:Liouville}

We recall that we consider positive and bounded solutions of
\Fi{s-a proof}
\left\{\begin{array}{ll}
Lu+f(x,u)=0, & x\in D,\\
u(x)\overset{\xi}{=}0, & x\in\partial D,
\end{array}\right.
\Ff 
where $f$ is assumed to satisfy the conditions in Assumption \ref{assum:standing}, and the meaning of the Dirichlet boundary condition $\overset{\xi}{=}$ was explained in Section \ref{section:Barrier}.

We define in this section
\[
\mathcal{L}:=L+c(x),
\]
which generalises the definition of $\L$ in the introduction. We prove the following theorems on non-existence at criticality and uniqueness. In the following $\ell>0$ is a parameter ($\ell\geq 1$ for our uniqueness theorem), and we assume $d\subseteq \Rm^d$ with $d\leq 2+\frac{4}{\ell}$.

\begin{theo}\label{thm:Liouv6}
	We assume that  Assumption \ref{assum:standing} holds. We further assume that there exists $c\in L^\infty(D)$ and $\sigma,\ell>0$ such that $d\leq 2+\frac{4}{\ell}$ and $f$ satisfies 
	\begin{equation}\label{proof f<f'-}
\text{$\ f(x,s)\leq c(x)s-\sigma s^{1+\ell}\ $ for all $\ s>0$, $x\in  D$.}
\end{equation}
Then if $\lambda(-(L+c))\geq 0$, the problem \eqref{s-a proof} does not admit any positive bounded solution.
\end{theo}
We are using $c$ to generalise $\partial_sf(x,0)$ as it is the former which we shall use to prove uniqueness of steady states (see Remark \ref{rmk:generalisation to c}).

We will need the condition \eqref{concav1} for our uniqueness theorem, which is:
\Fi{concav1 proof}\left\{\begin{array}{l} \text{there exists a positive constant
$\sigma>0$ such that}\\
\displaystyle  f(x,s_2)/s_2-f(x,s_1)/s_1\leq -\sigma (s_2-s_1)^{\ell}\\
\text{for all} \quad 0<s_1<s_2\leq S\quad\text{and}\quad x\in D.
\end{array}\right.
\Ff
We now assume that $\ell\geq 1$, so that \eqref{concav1 proof} can be satisfied.

\begin{theo}\label{theo:uniqueness proof}
We assume that we have Assumption \ref{assum:standing}, $\ell\geq 1$ and $ D\subseteq \Rm^d$ for $d\leq 2+\frac{4}{\ell}$. We further assume that $f$ satisfies \eqref{concav1 proof}. Then \eqref{s-a proof} admits at most one positive, bounded solution.
\end{theo}



\subsection{Proof of Theorem \ref{thm:Liouv6}}


We assume $D$ is unbounded; Theorem \ref{theo:nonexistence at criticality bded domains} covers the $D$ bounded case. By adding a constant to $c$ if necessary, we assume without loss of generality that $\lambda(-\L)=0$. 

Our proof is broken down into the following steps:
\begin{enumerate}
\item We assume that a solution of \eqref{s-a proof}, $u$, exists. We construct a sequence of solutions, $u_n$, of linearised versions of \eqref{s-a proof} on an increasing family of subdomains $(D_n)$, which we show converge to $u$.
\item Denoting the principal eigenfunction of $\L$ on $D_n$ by $\varphi_n$, we derive an inequality satisfied by $u_n$, $\varphi_n$, and their ratio. This is \eqref{eq:inequality 1 for proof}.
\item We define two quantities, $H_R$ and $U_R$, in terms of $u_n$ and $\varphi_n$ on $D_n\cap B_R$. We establish they satisfy an inequality, \eqref{HR3}. The dimension threshold is seen here, as the growth in $R$ of $\lvert D_n\cap B_R\rvert$ determines its coefficients.
\item We prove that along an appropriate subsequence, the coefficients in this inequality are suitably bounded. This is \eqref{eq:HR4}.
\item We use this to derive a quadratic inequality, which in turn we use to obtain bounds on $H_R$ and $U_R$. The dimension constraint comes out here in \eqref{eq:appearance of dim constraint}.
\item Having uniformly bounded $H_R$ and $U_R$, we feed these back into \eqref{eq:HR4} to see that $H_R$ can be made as small as one likes.
\item These controls on $H_R$ are uniform in $n$. By taking appropriate limits, we show that in fact $u$ must be $\L$-harmonic. This is a contradiction.
\end{enumerate}

\subsubsection*{Step 1}

We consider the same sequence of $C^{\infty}$ domains 
$(D_n)_{n\in\N}$ as in the proof of Proposition~\ref{prop:maximal solution} $(ii)$,
and the corresponding principal eigenvalues 
$\lambda(-(L+c);D_n)$ and eigenfunctions $\vp_n$.
We assume in addition that $0\in D_1$, which incurs no loss of generality by translation, and normalise the $\vp_n$ by $\varphi_n(0)=1$.
We denote $B_R:=B(0,R):=\{x\in \Rm^d:\lvert x\rvert<R\}$, for $R\in (0,\infty)$\commentout{\nota{\BS first time we
use this notation?}}. We write as in Section \ref{sec:uniq from non-exis},
\[
g(x,s):=f(x,s)/s,\quad x\in D,\; s>0.
\]

We assume for contradiction that $u>0$ is a positive and bounded solution 
of~\eqref{s-a proof}. 
In the following, $R_{\max}<\infty$ is to be determined, and assumed to be such that $\partial B_{R_{\max}}$ and $\partial D_n$ are transversal for all $n<\infty$. 

Since $D_n\subset\subset D$, $D$ is unbounded, $c(x,s)\geq f(x,s)/s$ and $\lambda(-(L+c);D)=0$, we have
\begin{equation}\label{eq:evalue pve}
\lambda(-(L+g(x,u(x)));D_n\cap B_{R_{\max}})> \lambda(-(L+g(x,u(x)));D)\geq   0.
\end{equation}

Thus there exists a unique solution, $u_n$, of the Dirichlet problem
    \[
    \left\{\begin{array}{ll}
(L+g(x,u(x))u_n=0, & x\in D_n\cap B_{R_{\max}},\\
u_n(x)=0, & x\in\partial D_n \cap B_{R_{\max}},\\
u_n(x)=u(x), & x\in D_n \cap \partial B_{R_{\max}},
\end{array}\right.
\]
Here we are considering $g(x,u(x))$ to be the zeroth order coefficient, so that the above is a linear elliptic equation with Dirichlet boundary conditions (along the piecewise $C^{\infty}$ boundary $\partial(D_n\cap B_{R_{\max}})$). 

In this proof, $\alpha$ is a H{\"o}lder exponent which is global, depends only upon $d$ and the ellipticity constants of $a$, but which may decrease from line to line (for instance in order to take convergent subsequences).

\begin{lem}\label{lem:convergence un to u}
We have $u_n\uparrow u\leq S$ in $C^{2,\alpha}_{\loc}$ as $n\ra\infty$.  
\end{lem}
\begin{proof}[Proof of Lemma \ref{lem:convergence un to u}]
It follows from the maximum principle \cite[Theorem 1.1]{BNV1994} (which we can apply by \eqref{eq:evalue pve}) that $0<u_n\leq u_m\leq u\leq S$ for $n<m$.

The pointwise limit must be a $C^{2,\alpha}_{\loc}$ limit by Schauder estimates. Therefore $u_{\infty}:=\lim_{n\ra\infty}u_n$ must be a solution of
    \[
    \left\{\begin{array}{ll}
(L+g(x,u(x))u_{\infty}=0, & x\in D\cap B_{R_{\max}},\\
u_{\infty}\overset{\xi}{=}0, & x\in\partial D \cap B_{R_{\max}},\\
u_{\infty}(x)=u(x), & x\in  D \cap \partial B_{R_{\max}},
\end{array}\right.
\]
Whereas the internal stability of Dirichlet boundary conditions formulated via barrier functions is standard \cite{BNV1994}, we establish this here for completeness. We take the diffusion corresponding to $L$, and the corresponding Feynman-Kac representation with zeroth-order term $g(x,u(x))$. We let $\tau_n:=\inf\{t>0:X_{t-}\in \partial (D_n\cap B_{R_{\max}})\}$, and $\tau$ the exit time from $D$. Then from the Feynman-Kac representation we have
\begin{equation}\label{eq:equation satisfied by u infty}
\begin{split}
u_n(x)&\leq Se^{\epsilon\lvert\lvert g(x,u(x))\rvert\rvert_{\infty}}[\Pm_x(\tau_n\leq \epsilon,X_{\tau_n-}\in \partial B_{R_{\max}}\cap D)+\Pm_x(\tau_n>\epsilon)]\\
&\leq Se^{\epsilon\lvert\lvert g(x,u(x))\rvert\rvert_{\infty}}[\Pm_x(\sup_{0\leq t\leq \epsilon}\lvert X_t\rvert\geq R_{\max})+\Pm_x(\tau>\epsilon)].
\end{split}
\end{equation}
Since the right-hand side is independent of $n$, it is an upper bound for $u_{\infty}$. By taking $\xi(x_n)\ra 0$ with $\sup_n\lvert x_n\rvert\leq R_{\max}$, then $\epsilon\ra 0$, we obtain that $u_{\infty}\overset{\xi}{=}0$ on $\partial D\cap B_{R_{\max}}$. 

Finally, $u$ also satisfies the same elliptic equation as $u_{\infty}$, \eqref{eq:equation satisfied by u infty}, whose solutions are unique by the maximum principle \cite[Theorem 1.1]{BNV1994} (using \eqref{eq:evalue pve}), hence $u_{\infty}=u$.
\end{proof}


\subsubsection*{Step 2}

We consider a cut-off function $\chi:\R\to\R$ which is smooth, nonincreasing, and satisfies
	\begin{equation}\label{eq:function chi cut-off}
	\chi(r)=1\quad\text{for }\;r\leq1,\qquad
	\chi(r)=0\quad\text{for }\;r\geq2.
	\end{equation}
	Then, for $R>0$, we call $\chi_R(x):=\chi(|x|/R)$, for $x\in\R^d$.
    
    Let us introduce some notation: we let $A$ denote the matrix field $(a_{ij})$, 
    which induces the scalar product $\langle p,q\rangle_A:=p\.A(x)q$ and the norm
	$|p|_A:=\sqrt{\langle p,p\rangle_A}$. 
    
    Our goal is to prove that 
	\begin{equation}\label{eq:inequality 1 for proof}
    \begin{split}
	2\int_{ D_n}\chi_R u_n\vp_n \langle \nabla\chi_R,\nabla w_n\rangle_A 
	+ \int_{ D_n} \chi_R^2\vp^2_n|\nabla w_n|^2_A\leq 
	-\sigma\int_{ D_n}\chi_R^2 u_n^{2+\ell}.
	\end{split}
    \end{equation}
	We derive the equation satisfied, in the weak sense, by the quotient $w_n=\frac{u_n}{\varphi_n}$, for test functions appropriate to our purposes. It is not completely straightforward because
	we consider test functions which are not in $H^1_0( D)$. 

	\begin{lem}\label{lem:weak}
		Suppose that $0<R<R_{\max}<\infty$, and $ D_n\cap B(0,R)\neq \emptyset$. Then the function 
			\[w_n:=\frac{u_n}{\vp_n}\]
		is continuous on $\bar D_n\cap B_{R_{\max}}$ and satisfies
		\[
			\int_{ D_n}\nabla v\. A(x)\vp^2_n\nabla w_n=
			\int_{ D_n}\big( g(x,u)u_n-c(x)u_n-\l(-\L; D_n)u_n\big)\vp_n v,
			\]	
		for any function $v\in H^1_{loc}( D_n)\cap L^\infty( D_n)$ supported on $B(0,R)$, and such that
		$\vp_n\nabla v\in L^\infty( D)$.
	\end{lem}
\begin{rmk}
Note in the above that $B_{R_{\max}}$ is the open ball of radius $R_{\max}$. Indeed, one should not, in general, expect $w_n$ to be continuous up to $\partial B_{R_{\max}}\cap \partial D_n$, due to the differing Dirichlet boundary conditions.
\end{rmk}
	
	\begin{proof}[Proof of Lemma \ref{lem:weak}]
		Firstly, we have that the function $w:=u_n/\vp_n$ is well defined and positive in $ D_n\cap B_{R_{\max}}$, and can be extended by continuity to $\bar D_n\cap B_{R_{\max}}$ because 
		$u,\vp\in C^1(\bar D_n\cap B_{R_{\max}})$, by boundary Schauder estimates \cite[Lemma 6.18, p.111]{Gilbarg1998}, and $\partial_\nu\vp_n>0$ on $\partial D\cap B_{R_{\max}}$, by Hopf's lemma. 
		
		The function 
		\[\vp^2_n \nabla w_n=\vp_n\nabla u_n-u_n\nabla\vp_n\]
		belongs to $H^1_{loc}(\bar D_n\cap B_{R_{\max}})\cap C^0(\bar D_n\cap B_{R_{\max}})$ and vanishes on $\bar D_n\cap B_{R_{\max}}$.
		It follows that $\vp^2_n$ satisfies, in the weak sense in $ D_n\cap B_{R}$,
		\[
		\begin{split}
		\nabla\.(A(x)\vp^2_n\nabla w_n) 
		&=\vp_n\nabla\.( A(x)\nabla u_n)-u_n\nabla\.( A(x)\nabla\vp_n)\\
		&=-\vp_n g(x,u)u_n+\big(c(x)+\l(-\L; D_n)\big)\vp_n u_n.
		\end{split}
		\]		
		Thus the conclusion of the lemma immediately holds if $v\in H^1_0( D_n)$, which proves the lemma
		in the case $ D_n=\R^d$.
		
		In the case $ D_n\neq\R^d$, we introduce a smooth cut-off function $\gamma:\R\to\R$  satisfying
			\[\gamma(r)=0\quad\text{for }\;r\leq0,\qquad
			\gamma(r)=1\quad\text{for }\;r\geq1,\]
		then, for $\delta>0$, we call 
		\[\gamma_\delta(x):=\gamma(d(x/\delta)),
		\]
		where $d(x)$ is the distance function from $\R^d\setminus D_n$. By the regularity assumption on $ D_n$,
		the function $\gamma_\delta$ is Lipschitz continuous provided that $\delta$ is sufficiently small.
		In the following we restrict to such values of $\delta$. In particular, we have that $v\gamma_\delta\in H^1_0( D\cap B(0,R_1))$, whence
		we can apply the above weak equation, namely,
		\[
	\int_{ D_n}\nabla(v\gamma_\delta)\.(A(x)\vp^2_n\nabla w_n) 
		=\int_{ D_n} \big( g(x,u)u_n-c(x)u_n-\l(-\L; D_n)u_n\big)\vp_n v\gamma_\delta.
		\]			
		The right-hand side converges to $\int_{ D_n} \big( g(x,u)u_n-c(x)u_n-\l(\L; D_n)u_n\big)\vp v$
		as $\delta\to0$, by dominated convergence.
		The left-hand side is 
		\[
		\int_{ D_n} \gamma_\delta\vp_n\nabla v\.(A(x)\vp_n\nabla w_n) +
		\int_{ D_n} v \nabla \gamma_\delta\.(A(x)\vp^2_n\nabla w_n).
		\]	
		The function $\vp_n\nabla v$ is bounded by hypothesis, while
		\[
		\vp_n\nabla w_n=\nabla u_n-u_n\nabla\vp_n\in C^0(\bar D_n\cap B_{R_{\max}}).
		\]	
		Thus the first term converges to $\int_{ D_n} \nabla v\.(A(x)\vp^2_n\nabla w_n)$ as $\delta\to0$,
		again by dominated convergence. The second one is bounded by
		\[
		\|\nabla \gamma_\delta\|_{L^1( D^\delta\cap B_R)}
		\|v\|_{L^\infty( D_n)}\|\vp^2_n\nabla w_n\|_{L^\infty( D^\delta\cap B(0,R))},
		\]
		where $ D^\delta:=\{x\in D_n\ :\ d(x)<\delta\}$.			
		We observe that
		\[
		\|\nabla \gamma_\delta\|_{L^1( D^\delta\cap B_R)}
		\leq\frac1\delta\|\gamma'\|_\infty | D^\delta\cap B_R|,
		\]	
		which remains bounded as $\delta\to0$, thanks to the regularity of $\partial D_n$. Finally,
		\[\|\vp^2_n\nabla w_n\|_{L^\infty( D^\delta\cap B_R)}\to0\qquad\text{as }\;
		\delta\to0.\]
		because $\vp_n\nabla w_n$ is continuous on $\bar D \cap B(0,R_{\max})$ and $\vp_n$ vanishes on $\partial D_n$. Thus we conclude Lemma \ref{lem:weak}.
%
%
	\end{proof}

	We apply Lemma~\ref{lem:weak} with $w_n:=u_n/\vp_n$ and $v:=\chi_R^2 w_n$, for $R\leq R_{\max}/2$, which is possible because $v$ is supported on $B_{2R}\subset\subset B_{R_{\max}}$ and
	\[\vp_n\nabla v=2\chi_Ru_n\nabla\chi_R+\chi_R^2(\nabla u_n-w_n\nabla\vp_n)\ \in L^\infty( D).\]
Here we don't denote the dependence of $v$ on $n$ as we will not need to keep track of this once we come to vary $n$.

We get
		\[
			\int_{ D_n}\nabla v\. A(x)\vp^2_n\nabla w_n=
			\int_{ D_n}\big( g(x,u)u_n-c(x)u_n-\l(-\L; D_n)u_n\big)\vp_n v,
			\]		

	We recall that the matrix field $A$ induces the scalar product $\langle p,q\rangle_A:=p\.A(x)q$ and the norm
	$|p|_A:=\sqrt{\langle p,p\rangle_A}$. Then we have
    \[
    \begin{split}
\nabla v\cdot A(x)\vp_n^2\nabla w_n&=w_n\varphi_n^2\nabla(\chi_R^2)\cdot A(x)\nabla w_n+\chi_R^2\nabla w_n\cdot A(x)\nabla w_n\\
&=2\chi_Ru_n\varphi_n\langle \nabla \chi_R,\nabla w_n\rangle_A+\chi_R^2\varphi_n^2\lvert\nabla w_n\rvert_A^2.
\end{split}
    \]
    We then rewrite the above, using $g(x,s)-c(x)\leq -\sigma s^{\ell}$, to get the inequality 
	\[
	\begin{split}
	&2\int_{ D_n}\chi_R u_n\vp_n \langle \nabla\chi_R,\nabla w_n\rangle_A 
	+ \int_{ D_n} \chi_R^2\vp_n^2|\nabla w_n|^2_A \\ &\quad\leq 
	-\sigma\int_{ D_n}\chi_R^2 u^{\ell}u_n^2-\underbrace{\lambda(-\mathcal{L}; D_n)}_{\geq \lambda(-\mathcal{L}; D)\geq 0}\int_{ D}\underbrace{\chi_R^2u_n^2}_{=u_n\vp_nv_n},\\
	&\quad \leq -\sigma\int_{ D_n}\chi_R^2 u^{\ell}u_n^2\leq -\sigma \int_{ D_n}\chi_R^2u_n^{2+\ell},
	\end{split}
	\]
    where, for the last inequality, we have used $0\leq u_n\leq u$.
Whence we have established~\eqref{eq:inequality 1 for proof}. 
We highlight that this is where we use $\lambda(-\L)\geq 0$ in our proof.

\subsubsection*{Step 3}	

Using the Cauchy-Schwarz inequality and the fact that $\nabla\chi_R\equiv0$ on $B_R$, we get
	\[
	\begin{split}
		\int_{ D_n} \chi_R^2\vp_n^2|\nabla w_n|^2_A &\leq 
	2\int_{ D_n}\chi_R u_n\vp_n |\nabla\chi_R|_A|\nabla w_n|_A 
	-\sigma\int_{ D_n}\chi_R^2 u_n^{2+\ell}\\
	& \leq 
	2\left(\int_{ D_n\setminus B_R}\chi_R^2 \vp_n^2 |\nabla w_n|_A^2\right)^{\frac12} 
	\left(\int_{ D_n} u_n^2 |\nabla\chi_R|_A^2\right)^{\frac12} 
	-\sigma\int_{ D_n\cap B_R}u_n^{2+\ell}
	\end{split}
	\]

	We now rewrite this inequality using the following notation:
	\[H_R:=\left(\int_{ D_n} \chi_R^2\vp_n^2|\nabla w_n|^2_A\right)^{\frac12},\qquad
	U_R:=\int_{ D_n\cap B_{R}}u_n^{2+\ell}.
	\]
	We do not denote the dependence of $H_R$ and $U_R$ on $n$ to keep the notation cleaner, but one should bear in mind that both of these depend on $n$.

	After observing that 
	\[\int_{ D_n\setminus B_R}\chi_R^2 \vp_n^2 |\nabla w_n|_A^2=
	\int_{ D_n}\chi_R^2 \vp_n^2 |\nabla w_n|_A^2-\int_{ D_n\cap B_R}\vp_n^2 |\nabla w_n|_A^2
	\leq H_R^2-H_{R/2}^2,
	\]
	we get
	\Fi{HR2}
		H_R^2 \leq 
		2\big(H_R^2-H_{R/2}^2\big)^{\frac12}
		\left(\int_{ D_n} u_n^2 |\nabla\chi_R|_A^2\right)^{\frac12} 
		-\sigma U_R.
	\Ff
	We then estimate the above integral by noticing that 
	$|\nabla\chi_R|_A\leq C R^{-1}$, for some $C>0$ only depending on $d,\chi$ and the ellipticity constants of $A$,
    \[
    \int_{ D_n} u_n^2 |\nabla\chi_R|_A^2
	\leq C^2R^{-2}\lvert\lvert u_n\rvert\rvert^2_{L^2( D_n\cap B_{2R})}.
    \]
    We now recall that for a finite measure space $(X,\mu)$ and $1\leq p<q\leq \infty$, one has $\lvert\lvert \cdot\rvert\rvert_{L^p(\mu)}\leq (\mu(X))^{\frac{1}{p}-\frac{1}{q}}\lvert\lvert \cdot\rvert\rvert_{L^q(\mu)}$. Using this with $p=2$ and $q=2+\ell$,
	\Fi{Holder}
	\begin{split}
    \int_{ D_n} u_n^2 |\nabla\chi_R|_A^2
	\leq C^2R^{-2}\lvert\lvert u_n\rvert\rvert_{L^{2+\ell}( D_n\cap B_{2R})}^2\lvert  D_n\cap B_{2R}\rvert^{\frac{\ell}{2+\ell}}\\
    = C^2R^{-2}\left(\int_{ D_n\cap B_{2R}} u_n^{2+\ell}\right)^{\frac{2}{2+\ell}}| D_n\cap B_{2R}|^{\frac{\ell}{2+\ell}}.
    \end{split}
    \Ff

Inputting this into \eqref{HR2}, we get
	\Fi{HR3}
		H_R^2 \leq 
		2CR^{-1}\lvert  D_n\cap B_{2R}\rvert^{\frac{\ell}{2(2+\ell)}}\big(H_R^2-H_{R/2}^2\big)^{\frac12}
		U_{2R}^{\frac{1}{2+\ell}}
		-\sigma U_R.
	\Ff

\subsubsection*{Step 4}

	We now need the following technical lemma.
	
	\begin{lem}\label{lem:technical lemma g}
		Let $g\in L^\infty(\R^d)$ be a non-negative, not a.e.~null function. Then
		\[\liminf_{R\to+\infty}\,\frac{\int_{B_{2R}}g}{\int_{B_{R}}g}\leq 2^d.\]	
	\end{lem}

	\begin{proof}[Proof of Lemma \ref{lem:technical lemma g}]
		We argue by contradiction. Thus we assume there exists $M>2^d$ and $R_0>0$ such that 
		\[	\int_{B_{2R}}g\geq M\int_{B_{R}}g,\]
		for all $R\geq R_0$. We deduce
		\[
		\int_{B_{2^kR_0}}g\geq M^k\int_{B_{R_0}}g,\]
		for all $k\in \N$. However, since $\lvert\lvert g\rvert\rvert_{L^{\infty}}<\infty$, there exists $C<\infty$ such that
		\[
		\int_{B_{2^kR_0}}g\leq \lvert\lvert g\rvert\rvert_{L^{\infty}}\lvert B_{2^kR_0}\rvert\leq C (2^kR_0)^d=CR_0^d2^{dk}.
		\]
		This is impossible because ${M>2^N}$.
	\end{proof}
	
We now fix arbitrary 
\[
M>2^d,\quad R_{\min}<\infty \quad\text{and}\quad K\in \N.
\]
Later we will vary $R_{\min}$ and $K$, hence we will keep track of how the constants we obtain depend upon them, but $M>2^d$ is fixed and unchanging, so we will ignore the dependence of constants we obtain on $M$. 

We apply the above lemma to the function  $\1_D u^{2+\ell}$ to find a sequence $(R_k)_{1\leq k\leq K}$ in $(R_{\min},\infty)$ such that 
	\[
	\int_{ D\cap B_{2R_k}}u^{2+\ell}< M\int_{ D\cap B_{R_k}}u^{2+\ell},
	\]
and $R_{k+1}>2R_k$, for all $1\leq k\leq K$. We now fix
\[
R_{\max}=10 R_{K},
\]
which thereby fixes $(u_n)_{n\geq 1}$.

Now by the monotone convergence theorem, we can choose $n_0<\infty$ such 
\[
\int_{ D_n\cap B_{2R_k}}u_n^{2+\ell}<M\int_{ D_n}u_n^{2+\ell}
\]
for all $1\leq k\leq K$ and all $n\geq n_0$. Then \eqref{HR3} implies
\begin{equation}\label{eq:HR4}
H_R^2\leq \Delta_R\big(H_R^2-H_{R/2}^2\big)^{\frac12}
		U_{R}^{\frac{1}{2+\ell}}-\sigma U_R,
\end{equation}
for all $R\in\{R_1,\ldots,R_K\}$ and $n\geq n_0$, where we define 
\[
\Delta_R:=2CM^{\frac{1}{2+\ell}}R^{-1}| D\cap B_{2R}|^{\frac{\ell}{2(2+\ell)}}.
\] 
Here $C>0$ is the constant in \eqref{HR3}, which depends only upon $d$, $\chi$ and the ellipticity constants of $A$.

\subsubsection*{Step 5}

    Now for $a>0$, $b,c\in \Rm$, there exists a solution of the quadratic inequality
    \begin{equation}\label{eq:quadratic ineq}
ax^2+bx+c\leq 0
    \end{equation}
    if and only if $b^2\geq 4ac$, and moreover any solution is bounded by the largest real root of the corresponding quadratic equation, so is at most
    \[
\frac{-b+\sqrt{b^2-4ac}}{2a}.
\]

We see from \eqref{eq:HR4} that
\[
H_R^2- U_{R}^{\frac{1}{2+\ell}}\Delta_R H_R+\sigma U_R\leq 0,
		\]
		hence $H_{R_k}$ satisfies \eqref{eq:quadratic ineq} with $a=1$, $b=-\Delta_{R_k}U_{R_k}^{\frac{1}{2+\ell}}$ and $c=\sigma U_{R_k}$. Thus
\[
H_{R_k}\leq \Delta_{R_k}U_{R_k}^{\frac{1}{2+\ell}}+\sqrt{\Delta_{R_k}^2U_{R_k}^{\frac{2}{2+\ell}}-\sigma U_{R_k}}
\]
and
\[
\Delta_{R_k}^2U_{R_k}^{\frac{2}{2+\ell}}-\sigma U_{R_k}\geq 0.
\]
We rewrite the latter inequality as
\[
U_{R_k}\leq \sigma^{-\frac{2+\ell}{\ell}}\Delta_{R_k}^{\frac{2(2+\ell)}{\ell}}.
\]

Now we have
\[
\Delta_R\leq C'M^{\frac{1}{2+\ell}}R^{\frac{d\ell}{2(2+\ell)}-1},
\]
for some constant $C'<\infty$ depending only upon $d$, $\chi$ and the ellipticity constants of $A$. Therefore 
\[
\frac{d\ell}{4+2\ell}-1\leq 0,
\]
i.e.
\begin{equation}\label{eq:appearance of dim constraint}
d\leq 2+\frac{4}{\ell},
\end{equation}
ensures that $\Delta_R$ is bounded, uniformly in $K$, $R_{\min}$ and $n$. We write $\bar \Delta$ for some uniform upper bound of $\Delta_R$.

Therefore there exists $C'<\infty$, which doesn't depend on $K$, $R_{\min}$ or $n\geq n_0$, such that
\[
H_{R_{k}},U_{R_{k}}\leq C'
\]
for all $1\leq k\leq K$ and $n\geq n_0$. 

\subsubsection*{Step 6}

But
\[
 \sum_{k=1}^{K-1}(H_{2R_K}-H_{R_K})\leq H_{R_{K}},
\]
hence there exists $k_0\in \{1,\ldots,K-1\}$ such that
\[
H_{2R_{k_0}}-H_{R_{k_0}}\leq \frac{C'}{K-1}.
\]
Therefore
\[
H_{2R_{k_0}}^2-H_{R_{k_0}}^2= (H_{2R_{k_0}}+H_{R_{k_0}})(H_{2R_{k_0}}-H_{R_{k_0}})\leq \frac{2C'^2}{K-1}.
\]

Inputting these back into \eqref{HR3}, we have
\[
		H_{R_{\min}}^2\leq H_{R_{k_0}}^2 \leq 
		\bar\Delta_{R_{k_0}}\big(\frac{2C'^2}{K-1}\big)^{\frac12}
		C^{\frac{1}{2+\ell}}.
		\]
This bound is uniform in $K$, $R_{\min}$ and $n\geq n_0$. 

\subsubsection*{Step 7}

Using Harnack's inequality and interior Schauder estimates, $\phi_n$ converges in $C^{2,\alpha}_{\loc}$ (on an appropriately chosen subsequence) to $\phi>0$, a solution of 
\[
\L\phi=0,
\]
on $ D$, with $\phi(0)=1$. Note that we make no claim about the boundary behaviour of $\phi$. We recall from Lemma \ref{lem:convergence un to u} that similarly $u_n\ra u$ in $C^{2,\alpha}_{\loc}$, along the whole sequence. 

Then for any $V\subset\subset B_{R_{\min}}\cap  D$, we have (by Harnack inequality and interior regularity)
\[
\left(\int_{V} \vp^2|\nabla (\frac{u}{\varphi})|^2_A\right)^{\frac12}\leq \limsup_{n\ra\infty}H^2_{R_{\min}}\leq \bar\Delta \big(\frac{2C'^2}{K-1}\big)^{\frac12}
		C^{\frac{1}{2+\ell}}.
\]

We have established that for all $\epsilon>0$ there exists $\L$-harmonic $\phi_{\epsilon}$ on $ D$, with $\phi_{\epsilon}(0)=1$, such that
\[
\int_{U} \vp^2_{\epsilon}|\nabla (\frac{u}{\varphi_{\epsilon}})|^2_A<\epsilon.
\]
Taking a $C^{2,\alpha}_{\loc}$ subsequential limit as $\epsilon\ra 0$ (again using interior Schauder estimates), we obtain another $\L$-harmonic function $\phi$, with $\phi(0)=1$, such that
\[
\int_{U} \vp^2|\nabla (\frac{u}{\varphi})|^2_A=0.
\]
But this implies that $\frac{u}{\varphi}$ is constant on $U$, hence $u$ is $\L$-harmonic. This is a contradiction, concluding the proof of Theorem \ref{thm:Liouv6}.\qed

\subsection{Proof of Theorem \ref{theo:uniqueness proof}}\label{sec:uniqueness}

We prove Theorem \ref{theo:uniqueness proof} by combining Theorem \ref{thm:Liouv6} with the machinery of Section \ref{sec:uniq from non-exis} (starting on Page \pageref{sec:uniq from non-exis}). 

We suppose for contradiction that there exists more than one solution of \eqref{s-a proof}. 
Then by Proposition \ref{prop:maximal solution} there exist
$0<u< \ol u\leq S$ solutions of \eqref{s-a}, with $\ol u$ being the maximal solution. 
The difference $w:=\ol u-u$ satisfies \eqref{eq:w-h} with $h$ 
defined by \eqref{h}. 
Proposition \ref{pro:nonex-uniqueness} implies that 
the function $h$ fulfils the conditions 
in Assumption \ref{assum:standing}, and also \eqref{proof f<f'-}
with $c(x)=f(x,u(x))/u(x)$. Moreover we see that
$\lambda(-L+c)\geq 0$ just by taking $u$ in the definition.
We obtain a contradiction from Theorem \ref{thm:Liouv6}.\qed

\section{Non-uniqueness in dimension $d\geq 2+\frac{4}{\ell}$}\label{section:infinite}

We consider the equation
\begin{equation}\label{eq:Hardy-critical}
	\Delta u + V(x)u = u^{1+\ell}\quad \text{ in }\;\Rm^d,
\end{equation}
with $V\in L^\infty(\R^d)$ being an asymptotic Hardy potential, that is,
\Fi{Hardy_potential}
\exists\e>0,\qquad
V(x) = \frac{c}{|x|^2} + o\left(|x|^{-2-\e}\right) \quad \text{ as }\; |x| \to \infty.
\Ff
Let us stress that we do not assume $V(x)$ to be radially symmetric, but only to have an asymptotic radial symmetry.

We impose $c$ to be smaller than the  critical Hardy constant 
\[H_d:=\left(\frac{d-2}{2}\right)^2.\]

\begin{theo}\label{thm:Hardy_spercritical}
	Let $\ell>0$, $d> 2+\frac{4}{\ell}$, and let
	$V$ be an asymptotic Hardy potential satisfying \eqref{Hardy_potential} with
	\Fi{cbounds}
	H_d-\Big(\sqrt{H_d}-\frac2{\ell}\Big)^2<c<H_d.
	\Ff
	Then the equation  \eqref{eq:Hardy-critical} admits infinitely many positive bounded solutions. 
	More precisely:
	\begin{enumerate}[$(i)$]
		\item \label{enum:evalue<=0 and both possible}
		$\lambda(-\L)\leq 0$, where $\L=\Delta +V(x)$ is the linearised operator around $0$, and both
		$\lambda(-\L)=0$ and $\lambda(-\L)<0$ can occur for suitable potentials $V$, 
        with the former always holding when $V\leq H_d/\lvert x\rvert^2$ everywhere;
		\item \label{enum:maximal soln O(x-2l)}
		the maximal bounded solution satisfies
		\[u_M(x)=\Theta(|x|^{-\frac 2\ell}) \quad \text{ as }\; |x| \to \infty;\]
		\item\label{enum:family uk}
		there is a family $(u_\kappa)_{\kappa>0}$ of positive solutions satisfying
		\[u_\kappa(x)=(\kappa+o(1))|x|^{a^*} \quad \text{ as }\; |x| \to \infty,\]
		with $a^*=-\sqrt{H_d}+ \sqrt{H_d-c}$;
		\item\label{enum:minimal soln estimate}
		if $\lambda(-\L)<0$ then the minimal solution satisfies
		\[u_m(x)=\Theta(|x|^{a_*}) \quad \text{ as }\; |x| \to \infty,\]
		with $a_*=-\sqrt{H_d}- \sqrt{H_d-c}$.
	\end{enumerate}
\end{theo}

\begin{rmk}
	Some comments on Theorems \ref{thm:Hardy_spercritical} are in order.
	\begin{enumerate}[$(a)$]
	\item
	 The condition $c<H_d$ in \eqref{cbounds} guarantess that the quantities $a_*,a^*$ are real and distinct.
	On the other hand the dimensional restriction $d> 2+\frac{4}{\ell}$ can be rewritten as $H_d>(2/\ell)^2$, which leaves room for $c$ to fulfil~\eqref{cbounds}. 
	Then, these conditions imply that
	\[a_*<a^*<-\frac2\ell.\]
	\item The appearance of the Hardy constant comes from Hardy's inequality, see \eqref{eq:use of Hardys inequality}.	
	\item When $d\leq 2+\frac{4}{\ell}$ our 
	construction  breaks down, and indeed we know from Theorems~\ref{theo:nonexistence at criticality} and \ref{theo:uniqueness} that in that case the equation \eqref{eq:Hardy-critical}
	does not admit a positive solution if $\lambda(-\L)=0$, and that it admits a unique positive solution if $\lambda(-\L)<0$.
	So a sharp dimensional change occurs when $d$ passes the threshold $2+\frac{4}{\ell}$ ($=6$ if $\ell=1$).
	\item	The one-parameter family of solutions to statement \eqref{enum:family uk} is a ghost of the linear space of the
	ground state of the linearised operator $L$ at infinity.
    \item In Section \ref{section:homogeneous counterexample}, we will use Theorem \ref{thm:Hardy_spercritical} to construct a locally (but not uniformly) analytic domain on which we have non-uniqueness for the classical homogeneous nonlinear term $u-u^2$ in $d\geq 7$, providing a counterexample to a conjecture of Berestycki and Graham in $d\geq 7$ \cite{Berestycki2025,Berestycki2026}.
	\end{enumerate}
\end{rmk}
When $\ell=1$, the lower bound on $c$ in \eqref{cbounds} is $H_d-(\sqrt{H_d}-2/\ell)^2=2d-8$. The following tells us this is sharp, and we have a double non-trivial phase transition.
\begin{theo}\label{theo:bound on solns l=1}
    Let $\ell=1$, $d> 6$, and let
	$V\in L^{\infty}(\Rm^d)$ be non-negative (not necessarily an asymptotic Hardy potential). Then the following hold:
    \begin{enumerate}
        \item \label{enum:V at most 2d-8}If $V\leq (2d-8)\lvert x\rvert^{-2}$ for all $x$ large enough, then there is at most one positive and bounded solution to \eqref{eq:Hardy-critical}, and such a solution exists if and only if $\lambda(-\L)<0$, where $\L:=\Delta+V$.
        \item \label{enum:V above Hd}If $V=\Theta(\lvert x\rvert^{-2})$ and, for some $c>H_d$, $V\geq c\lvert x\rvert^{2}$ for all $\lvert x\rvert$ large enough, then there is precisely one positive and bounded solution.
        \end{enumerate}
\end{theo}
\begin{rmk}[Double non-trivial phase transition]\label{rmk:double threshold PDE}
Take $V=\min(c\lvert x\rvert^{-2},1)$. There is a \textit{double non-trivial phase transition}: we have no positive and bounded solution for $c\in (0,2d-8)$, infinitely many for $c\in (2d-8,H_d)$, and exactly one for $c>H_d$. 
\end{rmk}

Before proving Theorems \ref{thm:Hardy_spercritical} and \ref{theo:bound on solns l=1}, we pose the following questions.

\begin{question}
It is unclear to us whether the solutions provided by \thm{Hardy_spercritical} are exhaustive. We point out however that \thm{Hardy_spercritical} (and its proof) imply that any positive, bounded solution must have a decay between
	$|x|^{a_*}$ and $|x|^{-2/\ell}$.
\end{question}
\begin{question}\label{question:sharp characterisation}
The spectral condition of \cite{Berestycki2025} fails to be satisfied whenever $V\ra 0$ at infinity (in fact, it fails more generally than this). In light of Theorems \ref{thm:Hardy_spercritical} and \ref{theo:bound on solns l=1}, this raises the question of whether a substantially sharper characterisation of uniqueness in dimensions above the threshold is possible.
\end{question}
In particular, take $f(s)=s-s^{1+\ell}$ for $\ell\geq 1$, $d>2+4/\ell$, and $r_0>0$ such that $\lambda(-(\Delta+1);(\Rm^d\setminus \cup_{x\in \mathbb{Z}^d}B(x,r_0)))=0$. Then in light of Theorem \ref{thm:Hardy_spercritical}, we posit that there exists $c\in \Rm_{>0}$ such that, taking $D=\Rm^d\setminus \cup_{x\in \mathbb{Z}^d}B(x,r_0-c/\lvert x\rvert^2)$, we have $\lambda(-(\Delta+1);D)=0$ but there are multiple positive and bounded solutions of 
\Fi{s-a counterexamples}
\left\{\begin{array}{ll}
Lu+f(u)=0, & x\in D,\\
u(x)=0, & x\in\partial D.
\end{array}\right.
\Ff

In order to prove \thm{Hardy_spercritical}, we first perform an analysis in a radially-symmetric framework.
	We introduce the radial operator ${\mathscr{L}}=\Delta+\frac{c}{|x|^2}$.
	When acting on radial functions $u=u(r)$ it becomes
\begin{equation*}
	{\mathscr{L}}=\partial_{rr}^2+\frac{d-1}{r}\,\partial_{r}+\frac{c}{r^2}.
\end{equation*}
	Note that all three terms are of the same order when computed on power-like functions. We also observe that this form is scale invariant.
	For 
	\[u(r)=r^a,\] we indeed get
	\[
	{\mathscr{L}} u=g(a)\,r^{a-2},
	\]
	where
	\[
	g(a):=a^2+(d-2)a+c.
	\]
	The sign of ${\mathscr{L}} u$ is determined by the sign of the polynomial $g$.
	The roots of~$g$ are real and distinct if and only if 
	\[c<H_d:=\Big(\frac{d-2}{2}\Big)^2,\]
	which is the second inequality in hypothesis \eqref{cbounds}. The roots are then
	\begin{equation*}
		a_*=-\sqrt{H_d}- \sqrt{H_d-c},
		\qquad
		a^*=-\sqrt{H_d}+ \sqrt{H_d-c}.
	\end{equation*}
	
	Next, we want the linear operator ${\mathscr{L}} u$ to dominate the nonlinear term $u^{1+\ell}$ in \eqref{eq:Hardy-critical}
	at infinity. This yields $a-2>a(1+\ell)$, that is, $a<-\frac2\ell$. We impose this on the largest root of $a^*$ and get
	\[
	-\sqrt{H_d}+ \sqrt{H_d-c}<-\frac2\ell,
	\]
	which amounts to the simultaneous conditions
	\[
	\sqrt{H_d}>\frac2\ell,
	\]
	i.e. $d> 2+\frac{4}{\ell}$, and
	\[
	c>\frac4\ell\sqrt{H_d}-\frac4{\ell^2}=H_d-\Big(\sqrt{H_d}-\frac2{\ell}\Big)^2.
	\]
	
In summary, under these conditions, which are in the hypotheses of the theorem, $g$
	has two real roots $a_*<a^*<0$, and
	\Fi{a*>nonlinear}
	a^*-2>a^*(1+\ell).
	\Ff
	
	After this preliminary analysis, we now proceed with the proofs of the statements of the theorem, in a different order.

	\begin{proof}[Proof of statement \eqref{enum:family uk}]
	For given $\kappa>0$, we want to construct a supersolution to~\eqref{eq:Hardy-critical} of the form
	\[w(r):=\kappa r^{a^*}+C r^a,\]
	with $a\in(a_*,a^*)$ and $C>0$ to be chosen.
	We have that
	\[\L w={\mathscr{L}} w+\Big(V(x)-\frac c{|x|^2}\Big)w=C g(a)r^{a-2}+\Big(V(x)-\frac c{|x|^2}\Big)w.\]
	In order to estimate the right-hand side for $r=|x|$ large we use \eqref{Hardy_potential}, so that
	there exists $D>0$ such that
	\Fi{asV}
	\big|V(x)-c|x|^{-2}\big|\leq D|x|^{-2-\e}
	\qquad\text{if }\;|x|\geq1.
	\Ff
	Therefore, the term $(V(x)-c|x|^{-2})w$ is at most of order $r^{a^*-2-\e}$, hence
	for $r$ large it is dominated by $g(a)r^{a-2}$ (which is negative because $a\in(a_*,a^*)$) 
	provided that $a>a^*-\e$. In conclusion, taking $a\in(\max(a_*,a^*-\e),a^*)$ and 
	$C$ sufficiently large, yields $\L w<0$ for $r\geq1$.
	On the other hand, constant functions with large enough constant are
	supersolutions to \eqref{eq:Hardy-critical}, and $w'(1)<0$. As a consequence, for $C$ large enough, the function
	\[\overline w(r):=\begin{cases}
		\kappa+C & \text{if }r\leq 1\\
		w(r) & \text{if }r> 1,
		\end{cases}
	\]
	is a generalised supersolution to \eqref{eq:Hardy-critical}.

    We now construct a subsolution of the form
	\[v(r):=\kappa r^{a^*}-C r^a,\]
	again with $a\in(a_*,a^*)$ and $C<\infty$ to be determined.
	Because of \eqref{a*>nonlinear}, we can take $a$ sufficiently close to $a^*$ so that
	\Fi{a>nonlinear}
	a-2>a^*(1+\ell).
	\Ff
	With the same computation as before we find
	\[\L  v\geq -C g(a)r^{a-2}-D\kappa r^{a^*-2-\e}
	\quad r>0.\]
	Hence if we impose as before 
	$a>a^*-\e$, and also $C> 1$, we get
	\[
	\L V>-g(a)r^{a-2}\quad\text{for}\quad r\geq \bar r:=[-2D\kappa/(Cg(a))]^{\frac{1}{a+\e-a^*}}.
	\]

So, in order to have a subsolution for $r\geq\bar r$, it is sufficient to have 
	that $-g(a)r^{a-2}$ is larger than the nonlinear term
	\[v^{1+\ell}=(\kappa r^{a^*}-C r^a)^{1+\ell},\] 
	in the region where the latter is non-negative, that is for $r$ larger than or equal to
	\[r_C:=\Big(\frac C\kappa \Big)^{\frac1{a^*-a}}=\text{ the unique $r$ such that $v(r)=0$}.\]

We will also want $r_C\geq \bar r$, which we see is guaranteed whenever $C$ is sufficiently large. We assume that $C<\infty$ is large enough to guarantee this. Then for $r\geq r_C$ we have $v^{1+\ell}\leq \kappa^{1+\ell}r^{a^*(1+\ell)}$, whence it is sufficient to have
\[
-g(a)r^{a-2}\geq \kappa^{1+\ell}r^{a^{\ast}(1+\ell)}\quad \text{for $r\geq r_C$.}
\]
	This condition is fulfilled for all $C$ sufficiently large because
	$r_C\to \infty$ as $C\to \infty$ and we have \eqref{a>nonlinear}. 
	
In summary, taking $C<\infty$ sufficiently large, we have guaranteed that $\L v\geq 0$ on $r\geq r_C$. We also have that $v(r_C)=0$, and $v'(r_C)=(a^{\ast}-a)\kappa r_C^{\ast}+ar_C^{-1}v(r_C)>0$. Thus if we adjust the function $v$ to be zero inside the ball $B(0,r_C)$, we obtain a generalised subsolution to \eqref{eq:Hardy-critical}.
	
	Finally, the sub/supersolution method yields a solution $v\leq u_\kappa\leq \overline w$, satisfying the desired decay.
\end{proof}

	\begin{proof}[Proof of statement \eqref{enum:maximal soln O(x-2l)}]  We first derive an upper bound for arbitrary solutions.
For given $C>0$, 
		we consider the family of radial functions $(w_\delta)_{\delta>0}$ defined by
	\[ w_\delta(r):=C r^{-\frac2\ell}+\delta.\]
	We compute, for $|x|=r>0$,
		\[\L  w_\delta= C \L (r^{-\frac2\ell})+\delta V(x)\leq
		C\tilde D r^{{-\frac2\ell}-2}+\delta \tilde D r^{-2},\]
	for some constant $\t D$ independent of $C,\delta$.
	Thus $w_\delta$ is a supersolution for $r> 0$ if 
	\[C\t D r^{{-\frac2\ell}-2}+\delta \t D r^{-2}
	\leq w_\delta^{1+\ell}=(C r^{-\frac2\ell}+\delta)^{1+\ell}.\]
	Since 
	\[(C r^{-\frac2\ell}+\delta)^{1+\ell}\geq (C r^{-\frac2\ell})^{1+\ell}+(1+\ell)C^\ell r^{-2}\delta
	= C^{1+\ell} r^{-\frac2\ell-2}+(1+\ell)C^\ell r^{-2}\delta,\]
	if we take $C=\t D^{1/\ell}$ we have that $w_\delta$ is a strict supersolution for any $\delta>0$. 
	Since a large constant $S>0$ is also a supersolution, 
	the function 
	\[
	\t w_\delta:=\min(w_\delta,S+\delta)=\min(Cr^{-\frac{2}{\ell}},S)+\delta
	\] is a bounded generalised strict supersolution, for any $\delta>0$.
	
We now use the family $(\t w_\delta)_{\delta>0}$ to show that any positive and bounded solution decays at least like $C|x|^{-\frac 2\ell}$. To do so, we fix $u_M$ to be the maximal positive and bounded solution, and prove that $w_{\delta}>u_M$ for all $\delta>0$.

We fix $\delta>0$. We firstly observe that, given an arbitrary sequence $(x_n)$ with $\lvert x_n\rvert\ra\infty$, $u_M(x_n+\cdot)$ converges in $C^{2,\alpha}_{\loc}$ (for some $\alpha\in (0,1)$) to a bounded and non-negative solution of $\Delta y-y^{1+\ell}=0$(since $V(x)\ra 0$ as $\lvert x\rvert\ra \infty$), the only solution of which is $y\equiv 0$. Therefore $u_M(x)\ra 0$ as $\lvert x\rvert\ra \infty$, so $w_{\delta}(x)>u_M(x)$ for all $\lvert x\rvert$ large enough.

We now prove that
\begin{equation}\label{eq:w delta geq uM}
w_{\delta}\geq u_M
\end{equation}
We suppose for contradiction that $\{x:w_{\delta}(x)<u_M(x)\}\neq \emptyset$. Then we can take $O$ to be a connected component of this set, which must be bounded by the above. Then on $O$ we have that $z:=u_M-w_{\delta}$ satisfies 
\[
\left(\L +\frac{f(x,u_M(x))-f(x,w_{\delta}(x))}{u_M(x)-w_{\delta}(x)}\right)z\geq 0,\quad z\overset{\xi}{=} 0\text{ on $\partial O$,}
\]
where $f(x,s):=V(x)s-s^{1+\ell}$ (recall the meaning of $\overset{\xi}{=}$ was defined in Section \ref{section:Barrier}). However, since $f_{ss}(x,s)<0$ for all $x\in \Rm^d$, $s>0$, we have
\[
\frac{f(x,u_M(x))-f(x,w_{\delta}(x))}{u_M(x)-w_{\delta}(x)}<\frac{f(x,u_M(x))}{u_M(x)},
\] 
for all $x\in \Rm^d$. Using this and the fact $O$ is bounded, we have
\[
\lambda(-(\L +\frac{f(x,u_M(x))-f(x,w_{\delta}(x))}{u_M(x)-w_{\delta}(x)}));O)>\lambda(-(\L +\frac{f(x,u_M(x))}{u_M(x)});\Rm^d)\geq 0.
\]
The last inequality holds here since $u_M$ is $[\L +f(x,u_M(x))/u_M(x)]$-harmonic. We obtain a contradiction from the maximum principle \cite[Theorem 1.1]{BNV1994}.

Thus $w_{\delta}\geq u_M$ everywhere, for all $\delta>0$. Taking $\delta\ra 0$, we obtain \[u_M\leq w_0=Cr^{-\frac2\ell}.\]\commentout{

	The family $(\t w_\delta)_{\delta>0}$ allows one to show that any bounded (sub)solution decays at least as $|x|^{-\frac 2\ell}$.
	Indeed, let $u$ be a bounded positive  solution.
	First of all, using that $V(x)\to0$ as $|x|\to\infty$, one readily sees that the same holds true for $u$\nota{\BS details needed?}.
	As a consequence, $\t w_\delta>u$ for any $\delta>0$, for $|x|$ large enough, and thus $\t w_\delta>u$ on $\R^d$
	provided that $\delta>0$ is sufficiently large.
	Suppose by contradiction that this inequality does not hold true for any $\delta>0$, and call $\bar\delta$ the maximal
	$\delta>0$ for which it fails somewhere in $\R^d$.
	This means that the strict supersolution $\t w_{\bar\delta}$ touches from above the solution $u$ somewhere, 
	contradicting the strong maximum principle.
	We have thereby shown that
	$u<\t w_\delta$ for all $\delta>0$, that is
	\[u\leq w_0=Cr^{-\frac2\ell}.\]}
	
	We need to show that the maximal solution satisfies a lower bound with the same power. 
	We achieve this by constructing a subsolution.
	We fix $a\in(a_*,a^*)$ and then, for $\eta>0$ to be determined, we set
	\[ v(r):=\eta r^{-\frac2\ell}-r^a.\]
	Using \eqref{asV}, we find for $|x|=r\geq1$ that
	\[\L  v\geq \eta r^{-\frac2\ell-2}\big( g(-2/\ell)-D r^{-\e}\big)
	-r^{a-2}\big(g(a)+D r^{-\e}\big).\]
	We now recall that $a_*<a<a^*<-\frac2\ell$, hence $g(a)<0<g(-2/\ell)$,
	and so there exists $\bar r>1$ which does not depend upon $\eta$ such that we have
	\[\L  v\geq \frac{\eta g(-2/\ell)}{2} r^{{-\frac{2}{\ell}}-2}\quad \text{for}\quad r\geq \bar r.
	\]
	On the other hand, in the region where $v\geq0$, i.e.~for $r$ larger than or equal to
	\[r_\eta:=\eta^{\frac1{a+\frac2\ell}}=\text{ the unique $r$ such that $v(r)=0$},\]
	which goes to $+\infty$ as $\eta\to0$, one has 
	\[v^{1+\ell}(r)\leq \eta^{1+\ell} r^{-\frac2\ell-2}.\]
	As a consequence, for $\eta$ sufficiently small one has that $r_\eta> \bar r$, and
	\[\L  v>v^{1+\ell}\quad \text{on $r\geq r_\eta$.}\]
Moreover $v'(r_\eta)=(-\frac{2}{\ell}-a)\eta r^{-\frac{2}{\ell}-1}+ar^{-1}v(r_{\eta})>0$.
	
	Therefore the function $\max(v,0)$ is a generalised subsolution to~\eqref{eq:Hardy-critical}.
	We can then find a solution $u$ between $\max(v,0)$ and a large constant $S>0$ (which is a supersolution).
	This shows that the maximal solution $u_M$ satisfies, for $r$ sufficiently large,
	\[u_M\geq v\geq \frac\eta2 r^{-\frac2\ell}.\]
	This concludes the proof of statement \eqref{enum:maximal soln O(x-2l)}.
\end{proof}
	

	\begin{proof}[Proof of statement \eqref{enum:minimal soln estimate}]
	
	Take $a_*-\e<a<a_*$, with $\e>0$ from hypothesis~\eqref{Hardy_potential}. Then define the function 
	\[w(r):=r^{a_*}-r^a,\]
	which is non-negative for $r\geq1$.
	Using \eqref{asV}, we find, for $|x|=r\geq1$,
	\[\L  w\leq D r^{a_*-2-\e}
	-r^{a-2}\big(g(a)-D r^{-\e}\big).\]
	Since $g(a)>0$, we see that $\L  w<0$ outside some ball $B(0,\bar r)$, with $\bar r>1$ large enough.
	This clearly holds true for any positive multiple of $w$. 
	We deduce that $Cw$ is a supersolution to~\eqref{eq:Hardy-critical} outside $B(0,\bar r)$, for any $C>0$.
	
	We then take $C$ large enough so that the constant function $C(\bar r^{a_*}-\bar r^a)$ is a supersolution to \eqref{eq:Hardy-critical}.
	Then, with this choice, the function $\min(Cw,C(\bar r^{a_*}-\bar r^a))$ is a generalised supersolution in the whole space $\R^d$.
	
	The condition ${\lambda(-\L )<0}$ allows one to obtain a minimal positive solution, $u_m$, which lies below any positive supersolution. We deduce that 
	\[
	u_m\leq C( r^{a_*}- r^a)\leq C r^{a_*}.
	\]


	We now derive a lower bound with the same power $a_*$, for any given positive solution $u$.
	We fix $a,b$ such that $a<a_*<b<a^*$, and in addition $a$ close enough to~$a_*$ so that
	\Fi{a>a_*-e}
	a>a_*-\e\quad \text{ and }\quad
	a-2>a_*(1+\ell),
	\Ff
 	which is possible because the second inequality holds for $a=a_*$, since $a_*<-\frac2\ell$.

	Then for $\eta,\theta\in(0,1]$, we define the function
	\[ v(r):=\eta(r^{a_*}+r^a)-\theta r^{b}.\]
	This function is non-negative for $r\leq r_{\eta,\theta}$, where
	\[r_{\eta,\theta}:=\inf\{r>0:v(r)\leq 0\}=\inf\{r>0:r^{a_*-b}+r^{a-b}\leq \frac\theta\eta\}\geq \Big(\frac{\eta}{\theta}\Big)^{\frac{1}{b-a^*}}.\]
	If $|x|=r\in[1,r_{\eta,\theta}]$, we have by 	\eqref{asV},
	\[
	\begin{split}
		\L  v-v^{1+\ell} \geq &\, \eta\big[-Dr^{a_*-2-\e}+r^{a-2}( g(a)-D r^{-\e})-\eta^\ell (r^{a_*}+r^a)^{1+\ell}\big]\\
		&
	-\theta r^{b-2}\big(g(b)+D r^{-\e}\big),
	\end{split}\]
The last term on the right-hand side is positive for all $r$ large enough, not dependent upon $\eta,\theta$, because $g(b)<0$. Here we recall that $D$ is such that $|V(x)-c|x|^{-2}|\leq D|x|^{-2-\e}$, for all $|x|\geq1$. The term
	\[-Dr^{a_*-2-\e}+r^{a-2}( g(a)-D r^{-\e})-\eta^\ell (r^{a_*}+r^a)^{1+\ell}\] 
	is dominated by the positive term $r^{a-2}g(a)$, independently of $\eta\in(0,1]$, owing to~\eqref{a>a_*-e}.
	As a consequence, 
	there exists $\bar r\geq 1$, independent of $\eta,\theta$, 
	such that $\L  v> v^{1+\ell}$ if $r\in[\bar r,r_{\eta,\theta}]$.
	
	We then fix $\eta\in(0,1]$ such that 
	\Fi{v<u}
	\eta( \bar r^{a_*}+ \bar r^a)<\min_{\partial B(0,\bar r)}u_m,
	\Ff
	so that $u_m>v$ on $ \partial B(0,\bar r)$ for any $\theta> 0$. We now establish that 
	\begin{equation}\label{eq:u geq v difference of balls}
	u_m\geq v \quad\text{on} \quad B(0,r_{\eta,\theta})\setminus B(0,\bar r), \quad\text{for any $\theta$ such that $r_{\eta,\theta}>\bar r$.}
	\end{equation}
	\begin{proof}[Proof of \eqref{eq:u geq v difference of balls}]
	We repeat the proof of \eqref{eq:w delta geq uM}. In particular, we assume for contradiction that $\{x\in B(0,r_{\eta,\theta})\setminus B(0,\bar r):u_m(x)< v(x)\}$ is non-empty, and take $O$ to be a connected component of it. Then since $u_m>v$ on $\partial B(0,r_{\eta,\theta})\cup \partial B(0,\bar r)$, $O\subset\subset\text{int}(B(0,r_{\eta,\theta})\setminus B(0,\bar r))$ and we have $v-u_m\overset{\xi}{=}0$ on $\partial O$. Since $v$ is a subsolution and $u_m$ is a solution, we can then apply the maximum principle \cite[Theorem 1.1]{BNV1994} as in the proof of \eqref{eq:w delta geq uM} to conclude a contradiction, hence \eqref{eq:u geq v difference of balls} is proven. 
	\end{proof}
	Finally, we send $\theta \ra 0$, using that $r_{\eta,\theta}\ra \infty$ as $\theta \ra 0$, to obtain $u_m(x)\geq \eta(\lvert x\rvert^{a_{\ast}}+\lvert x\rvert^a)$ for all $\lvert x\rvert\geq \bar r$. The statement is proven.\end{proof}	

\begin{proof}[Proof of statement \eqref{enum:evalue<=0 and both possible}]
Since there exists $R<\infty$ such that $V\geq 0$ for	$\lvert x\rvert \geq R$, we have
\[
\lambda(-\L ):=\lambda(-\L ;\Rm^d)\leq \lambda(-\L ;B(0,R)^c)\leq \lambda(-\Delta;B(0,R)^c)=0.
\]


	In order to find a potential $V$ for which $\lambda(-\L )<0$, it is sufficient to add to an arbitrary non-negative potential a compact perturbation which is $1$ on a sufficiently large ball~$B(0,R)$ (for example). This guarantees that $\lambda(-\L )=\lambda(-\L ;\Rm^d)\leq \lambda(-(\Delta+1);B(0,R))<0$, without affecting the hypothesis \eqref{Hardy_potential}.
	
	We now suppose that $V\leq H_d/\lvert x\rvert^2$. Then the Rayleigh quotient gives 
	\begin{equation}\label{eq:use of Hardys inequality}
	\begin{split}
	\lambda(-\L )&=\inf_{u\in H^1_0:\lvert\lvert u\rvert\rvert_{L^2}=1}\int_{\Rm^d}\lvert \nabla u\rvert^2-V(x)\lvert u\rvert^2\\
	&\geq \inf_{u\in H^1_0:\lvert\lvert u\rvert\rvert_{L^2}=1}\left[\int_{\Rm^d}\frac{H_d-H_d}{\lvert x\rvert^2}\lvert u\rvert^2\right]\geq 0,
\end{split}
	\end{equation}
	where we are using Hardy's inequality and $V\leq H_d/\lvert x\rvert^2$ to obtain the first inequality.
	\begin{rmk}\label{rmk:Hardy constant sharp}
	The Hardy constant is sharp. From this it is simple to deduce that if $V\sim c/\lvert x\rvert^2$ for $c>H_d$, then $\lambda(-\L )<0$.
	\end{rmk}
    \end{proof}   
    \begin{proof}[Proof of Part \ref{enum:V at most 2d-8} of Theorem \ref{theo:bound on solns l=1}] 
    We take $\ell=1$, $d>6$ and suppose that $V=\min(c/\lvert x\rvert^2,A)$ for some $A<\infty$ and $2d-8<c<H_d$. Since $V$ is rotationally symmetric, so too is the maximal solution, $u_M$; in the following we write everything in polar coordinates. We claim that
    \begin{equation}\label{eq:bound on uM c large than 2d-8}
    u_M(r)\leq (8-2d+c)r^{-2}.
    \end{equation}
   We consider for $\alpha>1$ the rescaled equation,
    \begin{equation}\label{eq:rescaled equation R}
    \Delta u_\alpha+\frac{d-1}{r}\partial_ru_\alpha+V_\alpha(r)u_\alpha-u_\alpha^2=0,
    \end{equation}
    where $V_\alpha(r):=\alpha^2V(\alpha r)=\min(c/\lvert x\rvert^2,\alpha ^2A)$, whose solutions are related to those of \eqref{eq:Hardy-critical} by $u_\alpha (r)=\alpha ^2u(\alpha r)$. The maximal solution of the rescaled equation, $u_{\alpha ,M}$, satisfies $u_{\alpha ,M}(r)=\alpha ^2u_M(\alpha r)$. Then since $V_\alpha $ is non-decreasing in $\alpha $, so too is $u_{\alpha ,M}$, hence $\alpha ^2u_M(\alpha r)$ is non-decreasing in $\alpha $ for any $r\in (0,\infty)$. Since $\alpha ^2u_M(\alpha r)$ is bounded (by Statement \eqref{enum:maximal soln O(x-2l)} of Theorem \ref{thm:Hardy_spercritical}), there exists $A\in (0,\infty)$ such that $\alpha ^2u_M(\alpha r)\uparrow A r^{-2}$. On the other hand, from Schauder estimates applied to the $u_{M,\alpha }$s, $u_{\infty}:=\lim_{\alpha \ra\infty}u_{\alpha ,M}(r)=\lim_{\alpha \ra\infty}\alpha ^2u_M(\alpha r)$ satisfies
    \[
    \Delta u_{\infty}+\frac{d-1}{r}\partial_ru_{\infty}+\frac{c}{r^2}u_{\infty}-u_{\infty}^2=0,\quad r>0.
    \]
    Therefore we have \eqref{eq:bound on uM c large than 2d-8}. 

    Therefore if $V$ is any potential satisfying $V(x)\leq (2d-8)\lvert x\rvert^{-2}$ everywhere, then the comparison principle and the fact that the upper bound converges to $0$ as $c\downarrow 2d-8$ implies there is no positive and bounded solution to \eqref{eq:Hardy-critical}. 

    We now suppose that $V\in L^{\infty}(\Rm^d;\Rm_{\geq 0})$ satisfies $V\leq (2d-8)\lvert x\rvert^{-2}$ for all $\lvert x\rvert$ large enough. By using the rescaling in \eqref{eq:rescaled equation R} but with $0<\alpha \ll 1$, we assume without loss of generality that $V\leq 1$ everywhere. Note that whether $\lambda(-\L)<0$ or $\lambda(-\L)=0$ remains unchanged under this rescaling.
    
    We will need, in the rest of the proof of Theorem \ref{theo:bound on solns l=1}, to make use of the probabilistic correspondence introduced in the following section, recalled from \cite{Maillard2025} as Theorem \ref{theo:PDE survival connection}. We take in the following the setup of that section. Then for the corresponding branching Brownian motion $\bar X_t$, we need to show that $\Pm_x(\glSurv\setminus \locSurv)\equiv 0$. Theorem \ref{theo:PDE survival connection} and the above non-existence result tells us that we would have $\Pm_x(\glSurv)\equiv 0$ were we to know that $V\leq c\lvert x\rvert^{-2}$ everywhere. Therefore, since there exists $R<\infty$ and $c<2d-8$ such that $V(x)\leq c\lvert x\rvert^{-2}$ for $\lvert x\rvert\geq R$, we have $\Pm_x(\glSurv\cap \text{no particle returns to $B(0,R)$})\equiv 0$. This implies that $\Pm_x(\glSurv\setminus \locSurv)\equiv 0$.

    Finally, $\Pm_x(\glSurv)=\Pm_x(\locSurv)$, so $\Pm_x(\glSurv)>0$ if and only if $\lambda(-\L)<0$ by \cite{Englander2004}. This implies the existence of a positive and bounded solution is equivalent to $\lambda(-\L)<0$ by Theorem \ref{theo:PDE survival connection}.
    \end{proof}
    \begin{proof}[Proof of Part \ref{enum:V above Hd} of Theorem \ref{theo:bound on solns l=1}]  
    We have $\lambda(-\L)<0$ by Remark \ref{rmk:Hardy constant sharp}, so there exists a minimal positive and bounded solution by Theorem \ref{theo:PDE survival connection} and \cite{Englander2004}. Rescaling as above, we assume without loss of generality that $0\leq V\leq 1$ everywhere.

    The maximal solution satisfies $u_M(x)=O(\lvert x\rvert^{-2})$, by the proof of the upper bound in Part \eqref{enum:maximal soln O(x-2l)} of Theorem \ref{thm:Hardy_spercritical}. On the other hand, by following the proof of Part \ref{enum:V at most 2d-8} of Theorem \ref{theo:bound on solns l=1} with the minimal solution in place of the maximal solution, we see that the minimal solution satisfies $u_m(x)=\Theta(\lvert x\rvert^{-2})$. Thus there exists $C<\infty$ such that $u_M\leq Cu_m$ everywhere. 

    We take the corresponding branching Brownian motion $\bar X_t$, and suppose $\Pm_x(\glSurv\setminus\locSurv)>0$. Theorem \ref{theo:PDE survival connection} and Levy's upwards theorem imply that
    \[
    1-\prod_{i=1}^{N_t}(1-u_M(X^i_t))=\Pm_x(\glSurv\lvert \mathcal{F}_t)\ra 1\quad\text{and}\quad 1-\prod_{i=1}^{N_t}(1-u_m(X^i_t))=\Pm_x(\glSurv\lvert \mathcal{F}_t),
    \]
    almost surely on the event $\glSurv\setminus \locSurv$, where $(\mathcal{F}_t:t\geq 0)$ is the Filtration. However $\inf_i\lvert X^i_t\rvert \ra \infty$ on $\glSurv\setminus \locSurv$. Thus $\sum_{i=1}^{N_t}u_M(X^i_t)\ra \infty$ and $\sum_{i=1}^{N_t}u_m(X^i_t)\ra 0$ as $t\ra\infty$ on $\glSurv\setminus \locSurv$, but $\sum_{i=1}^{N_t}u_M(X^i_t)\leq C\sum_{i=1}^{N_t}u_m(X^i_t)$. This is a contradiction, so $\Pm_x(\glSurv\setminus\locSurv)\equiv 0$ and we are done by 
    Theorem~\ref{theo:PDE survival connection}.
    \end{proof}

\section{Proof of the results for the evolution 
equation}\label{section:evolution proofs}

We start this section with the proof of Theorem \ref{thm:cptperturb} $(i)$.
Then we will deduce as a corollary the stability of any positive
bounded steady state with respect to compact perturbations, 
and next we will derive from it
Theorem \ref{thm:cptperturb} $(ii)$. Finally, we construct the
non-convergent solution of Theorem \ref{thm:oscillating}.

\begin{proof}[Proof of Theorem \ref{thm:cptperturb} $(i)$]
Let $\vp$ be a generalised principal eigenfunction of $-\L$ 
in $D$ with Dirichlet boundary condition, that is
\[
\left\{
\begin{array}{rl}
-\L\varphi=\lambda(-\L)\vp
 & \text{in }D,\\[5pt]
\vp=0 & \text{on }\partial D.
\end{array}\right.
\]
This is provided by 
\cite[Theorem~1.4]{BR4}
under the assumption that $\partial D$ is $C^{1,1}$, 
but the construction works without any change for $\partial D$ 
locally Lipschitz
(the stronger regularity was only required in \cite{BR4}
to get some boundary regularity for $\vp$)\commentout{
\nota{Maybe we will need to rewrite how $\vp$ is constructed to justify the 
following integration}}. We then define
\[\gamma(t):=\int_D\eta(t,x)\vp(x)dx.\]
Using integration by parts and $\l(-\L)\geq0$, we find
\[\gamma'(t)=\int_D\partial_t\eta(t,x)\vp(x)dx
=\int_D\eta(t,x) L\vp(x)+f(x,\eta)dx
\leq\int_D \big(f(x,\eta)-f_s(x,0)\eta\big)dx.\]
It then follows form the KPP-type hypotheses that $\gamma'\leq0$.
Then in particular $\gamma(t)$ is bounded for $t\in[0,+\infty)$, 
and so is $\eta(t,x)$, locally uniformly in $x\in D$, thanks to the interior
Harnack inequality (see \cite{Krylov1981} or \cite[Corollary 7.42]{Lie}).

Suppose by contradiction that $\eta$ does not converge to $0$ as $t\to+\infty$, 
locally uniformly in $\bar D$. Then there exist $R,\varepsilon>0$ and a sequence $(t_n)_{n\in\N}$ diverging to $+\infty$ such that $\sup_{x\in D\cap B_R}\eta(t_n,x)>\varepsilon$
for all $n\in\N$.
On the other hand, we can find a set $K\subset\subset D\cap E$ with positive 
measure. Since $\partial\O$ is Lipschitz,
we can apply the boundary Harnack inequality
(see \cite{bHarnack-div} or \cite[Theorem 3.5]{Huska}) 
together with a covering argument
and the standard interior Harnack inequality, and deduce the existence of a 
positive constant $C$ such that
\[
\forall n\in\N,\qquad
\sup_{x\in D\cap B_R}\eta(t_n,x)\leq C\inf_{\tau\in[1,2],\,
x\in K}\eta(t_n+\tau,x).\]
Since the left-hand side above is larger than $\e$, we deduce 
\[\forall n\in\N,\ t\in[t_n+1,t_n+2],\quad
\gamma'(t)\leq 
|K|\sup_{x\in K,s\in[C^{-1}\e,S']}(f(x,s)-f_s(x,0)s).\]
where $S':=\|u\|_{L^\infty((0,+\infty)\times K)}$.
The right-hand side above is a negative constant, independent of $n$,
thus $\gamma(t)\to-\infty$ as $t\to+\infty$, which is impossible
because ${\gamma\geq0}$.
\end{proof}

\begin{cor}\label{cor:stability}
Suppose that Assumption~\ref{assum:standing_Lip} holds
    and that $f$ satisfies the KPP-type condition
    \eqref{eq:mediumKPP f}, together with the strong-KPP
    condition~\eqref{eq:strongKPP f x} for all $x$ in some set $E\subseteq  D$ with positive measure.
Suppose further that the stationary problem \eqref{s-a}
admits a positive bounded solution $u$.
Let $\eta$ be a solution of~\eqref{eq:general reaction-diffusion}
with an initial datum 
$\eta_0\geq0$ such that $\eta_0-u$ has compact support in $\bar D$.
Then $\eta(t,x)\to u(x)$ as $t\to+\infty$, locally uniformly in $x\in\bar D$.
\end{cor}

\begin{proof}
Suppose first that $\eta_0\geq u$.
Consider the difference $z(t,x):=\eta(t,x)-u(x)$. It 
is non-negative by the parabolic comparison principle, and satisfies
\[
\partial_t z-\partial_i(a_{ij}(x)\partial_j z)
=g(x,u(x))z+\big(g(x,u(x)+z)-g(x,u(x))\big)(u(x)+z),
\]
where $g(x,s):=f(x,s)/s$, which is nonincreasing in $s>0$, and strictly
decreasing for $x\in E$, by hypothesis.
Call
\[
h(x,s):=g(x,u(x))s+\big(g(x,u(x)+s)-g(x,u(x))\big)(u(x)+s)\frac{s}{\|z\|_\infty}.
\]
We see that $z$ is a subsolution to the equation
\Fi{RDh}
\partial_t \zeta-\partial_i(a_{ij}(x)\partial_j \zeta)= h(x,\zeta).
\Ff
One can check that $h$ satisfies the conditions on $f$
in Assumption~\ref{assum:standing_Lip},
and also \eqref{eq:weakKPP f}
with a strict inequality for all $x\in E$, where $\partial_s h(x,0)=g(x,u(x))$.
Moreover the linearised operator around $0$ for~\eqref{RDh}
is $\L_u:=L+g(x,u(x))$ which satisfies $\lambda(-\L_u)\geq 0$, as it
is seen by taking $u$ in its definition.
It then follows from Theorem~\ref{thm:cptperturb} $(i)$
that $z\to0$, i.e.~$\eta\to u$, 
locally uniformly in $\bar D$ as $t\to+\infty$.

Next, suppose that $\eta_0\leq u$, and call
now $z:=u(x)-\eta(t,x)\geq0$. 
For convenience, we extend $f$ for negative values of $s$ by
$f(x,s)=f_s(x,0)s-s^2$, so that the
strong-KPP condition~\eqref{eq:strongKPP f x} holds true for $s<0$,
for all $x\in D$.
The function $z$ satisfies
\[
\partial_t z-\partial_i(a_{ij}(x)\partial_j z)
=g(x,u(x))z+\big(g(x,u(x))-g(x,u(x)-z)\big)(u(x)-z),
\]
hence it is a subsolution to \eqref{RDh} with
\[
h(x,s):=g(x,u(x))s+\big(g(x,u(x))-g(x,u(x)-s)\big)(u(x)-s)\frac{s}{\|z\|_\infty}.
\]
This function $h$ fulfils the same conditions as the previous one, 
thus we conclude again that
$\eta\to u$ locally uniformly in $\bar D$ as $t\to+\infty$.

The general case without any sign condition on $\eta_0-u$
is derived by applying the above conclusions to the initial data
$\max(\eta_0-u,0)$ and $\min(\eta_0-u,0)$, and then using 
the parabolic comparison principle.
\end{proof}

\begin{proof}[Proof of Theorem \ref{thm:cptperturb} $(ii)$]
According to Proposition \ref{prop:maximal solution} $(ii)$, the stationary problem
admits a minimal positive solution $\underline u$. 
We also know from the proof of that
proposition that this minimal solution is obtained as the long-time limit
of a solution to~\eqref{eq:general reaction-diffusion}
with an initial datum of the form $\e\vp'$, for any $\e>0$ smaller than some $\e_0$,
where $\vp'$ is a continuous function positive
on some smooth domain $D'\subset\subset D$ and vanishing outside.
%
Since $\eta(1,\cdot)>0$ in $D$ by the strong maximum principle,
we can take $\e\in(0,\e_0)$ small enough so that $\e\vp'<\eta(1,\cdot)$ 
in $D'$.
We then get by comparison:
\[\liminf_{t\to+\infty}\eta(t,x)\geq \ul u(x),\]
locally uniformly for $x\in\bar D$.

On the other hand, consider the solution $\overline\eta$ 
to~\eqref{eq:general reaction-diffusion} with initial datum
$\max(\eta_0,u)$. Since $\overline\eta(0,\.)-u$ 
has compact support in $\bar D$, Corollary \ref{cor:stability}
yields $\overline\eta(t,x)\to u(x)$ as $t\to+\infty$, locally uniformly in $x\in\bar D$. It then follows by comparison:
\[\limsup_{t\to+\infty}\eta(t,x)\leq \lim_{t\to+\infty}\overline\eta(t,x)
\leq u(x),\]
locally uniformly for $x\in\bar D$.
The proof is concluded.
\end{proof}

\begin{proof}[Proof of Theorem \ref{thm:oscillating}]
Under both assumptions $a)$ and $b)$ we have that 
the stationary problem \eqref{s-a} admits a positive solution,
hence by Proposition \ref{prop:maximal solution} $(i)$
it admits a maximal positive solution $v$. 
In the case $a)$ there is another positive solution of \eqref{s-a}
which is smaller than $v$, that we call $u$, 
while in the case $b)$ we set $u\equiv0$. 




The initial datum of the non-convergent solution $\eta$ we construct 
will be of the form
\[
\eta_0(x):=
\begin{cases}
v(x) & \text{if $x\in D\cap B_{S_0}$}\\
u(x) & \text{if $x\in D\cap B_{R_n}\setminus B_{S_{n-1}}$
\ with $n\in\N$} \\
v(x) & \text{if $x\in D\cap B_{S_n}\setminus B_{R_n}$
\ with $n\in\N$},
\end{cases}
\]
for some suitable positive sequences
$(R_n)_{n\in\N}$, $(S_n)_{n\in\N}$  diverging to $+\infty$ and~satisfying
\[
\forall n\geq1,\qquad
n-1\leq S_{n-1}<R_n<S_n.
\]

In order to construct the sequences $(R_n)_{n\in\N}$, $(S_n)_{n\in\N}$,
we define, for $n\in\N$, 
\[
v^{n}(x):=
\begin{cases}
\eta_0(x) & \text{if $x\in D\cap B_{R_n}$} \\
v(x) & \text{if $x\in D\setminus B_{R_n}$}
\end{cases}
\]
and
\[
u^{n}(x):=
\begin{cases}
\eta_0(x) & \text{if $x\in D\cap B_{S_n}$} \\
u(x) & \text{if $x\in D\setminus B_{S_n}$},
\end{cases}
\]
and call $\overline \eta^n$, $\underline \eta^n$ the solutions of~\eqref{eq:general reaction-diffusion}
with initial data $v^n$, $u^n$ respectively.
By comparison we have 
\Fi{u<eta<v}
u(x)\leq \underline \eta^n(t,x)\leq \eta(t,x)
\leq \overline \eta^n(t,x)\leq v(x)
\quad\text{ for all }\; t\geq0,\ x\in D.
\Ff
Moreover $\overline \eta^{n}(t,x)\to v(x)$ and $\underline\eta^{n}(t,x)\to u(x)$
 as $t\to+\infty$, locally uniformly in $\bar D$, owing to
Corollary \ref{cor:stability} and Theorem~\ref{thm:cptperturb} $(i)$
in the case $b)$.

We claim that $(R_n)_{n\in\N}$, $(S_n)_{n\in\N}$ can be chosen in such a way that
\Fi{eta-u}
\sup_{x\in D\cap B_n}
\big(\overline \eta^n(\tau_n,x)-u(x)\big)\leq\frac{\|v\|_{L^\infty(D)}}n
\Ff
and 
\Fi{v-eta}
\sup_{x\in D\cap B_n}
\big(v(x)-\underline \eta^n(\sigma_n,x)\big)\leq\frac{\|v\|_{L^\infty(D)}}n,
\Ff
for some $\tau_n,\sigma_n$ satisfying
\[\sigma_n\geq\tau_n\geq n.\]
We prove the claim by induction on $n$.
The properties \eqref{eta-u}-\eqref{v-eta} automatically hold
for $n=1$, for any choice of $S_0,R_1,S_1$ and any $\tau_1,\sigma_1$.
Suppose that we have constructed the 
$R_j,S_j,\tau_j,\sigma_j$ with the desired properties up to some $j=n$.
Since $\underline \eta^{n}\to u$ as $t\to+\infty$, locally uniformly in $\bar D$,
we can find $\tau_{n+1}$ large enough, and in particular larger than
$n+1$, so that
\[
\sup_{x\in D\cap B_{n+1}}
\big(\underline \eta^n(\tau_{n+1},x)-u(x)\big)
<\frac{\|v\|_{L^\infty(D)}}{n+1}.
\]
Observe that the initial datum $v^{n+1}$ can be rewritten as
\[
v^{n+1}(x)=
\begin{cases}
u^n(x) & \text{if $x\in D\cap B_{R_{n+1}}$} \\
v(x) & \text{if $x\in D\setminus B_{R_{n+1}}$}. \\
\end{cases}
\]
Thus parabolic estimates show that the emerging
solution $\overline \eta_{n+1}$
converges to the solution $\underline \eta^n$ (emerging from $u^n$)
as $R_{n+1}\to+\infty$,
locally uniformly in time and space, whence for $R_{n+1}$ sufficiently large
there holds
\[
\sup_{x\in D\cap B_{n+1}}
\big(\overline \eta_{n+1}(\tau_{n+1},x)-u(x)\big)\leq
\frac{\|v\|_{L^\infty(D)}}{n+1},
\]
which is the desired property \eqref{eta-u} for $n+1$.
We can assume without loss of generality that $R_{n+1}>S_n$.

Analogously, using first that
$\overline \eta^{n+1}\to v$ as $t\to+\infty$, locally uniformly in $\bar D$,
we can find $\sigma_{n+1}\geq n+1$ in such a way 
that 
\[
\sup_{x\in D\cap B_{n+1}}
\big(v-\overline \eta^{n+1}(\sigma_{n+1},x)\big)<\frac{\|v\|_{L^\infty(D)}}{n+1}.
\]
Next, since 
\[
u^{n+1}(x)=
\begin{cases}
v^{n+1}(x) & \text{if $x\in D\cap B_{S_{n+1}}$} \\
u(x) & \text{if $x\in D\setminus B_{S_{n+1}}$}, \\
\end{cases}
\]
we can eventually choose $S_{n+1}$ large enough in such a way that
\eqref{v-eta} holds for $n+1$, and it is not restrictive to assume that $S_{n+1}>\max(R_{n+1},n+1)$.

We have therefore proved the claim.
Thus, we deduce from \eqref{u<eta<v} and \eqref{eta-u}-\eqref{v-eta} that
$\eta(\tau_n,x)\to u(x)$ and $\eta(\sigma_n,x)\to v(x)$ as $n\to\infty$,
locally uniformly in $x\in D$, which concludes the proof.
\end{proof}

\section{Proof of the results on the 
recurrence of BBM}\label{section:probabilistic proofs}

We recall that we consider a ($\sqrt{2}$ diffusivity) branching Brownian motion (BBM) $\bar X_t=(X^1_t,\ldots,X^{N_t}_t)$ on a domain $D\subseteq \Rm^d$, with branching mechanism specified by the branching rate $r(x)$ and offspring distribution $(p_n(x):0\leq n\leq n_{\max})$ ($n_{\max}<\infty$), satisfying the standing assumption, Assumption \ref{assum:standing prob}, which we restate here. 
\begin{assum}\label{assum:standing prob proof}
The domain $D\subseteq \Rm^d$ is non-empty, open, and connected. The branch rate $r(x)$ and offspring distribution $(p_n(x),\; 0\leq n\leq n_{\max})$  ($n_{\max}<\infty$) are locally H{\"o}lder continuous and satisfy \eqref{eq:pn equiv 1} and \eqref{eq:non-trivial branching}, namely
\[
\sum_{n=0}^{n_{max}}p_n(x)\equiv 1\quad \text{and}\quad \lvert \{x\in D:r(x)\sum_{n=2}^{n_{\max}}p_n(x)>0\}\rvert>0.
\]
\end{assum}

Note that we are imposing no assumptions on the boundary regularity, and the offspring distribution is an arbitrary bounded position-dependent offspring distribution. Due to the former, Dirichlet boundary conditions will be formulated in terms of barrier functions as in Section \ref{section:Barrier}. We shall rely on the results of \cite{Englander2004} and \cite{Maillard2025}. We note that the former restricts its attention to purely binary branching, whilst the latter imposes a $C^{1,1}$ boundary (to stay within the framework of \cite{BR4}, and for straightforwardness). Notwithstanding this, one can check that the proofs of all the results we will employ from \cite{Englander2004,Maillard2025} work under Assumption \ref{assum:standing prob proof}, with the additional ingredient of Proposition \ref{prop:xi vanishes a.s.}.

We firstly recall \cite[Eqns. (20) and (24)]{Maillard2025} that 
\begin{equation}\label{eq:pve somewhere everywhere equiv pf}
\begin{split}
&\text{$\Pm_x(\glSurv)>0$ or $\Pm_x(\locSurv)>0$ for some $x\in D$ if and only if}\\
&\text{$\Pm_x(\glSurv)>0$ (respectively $\Pm_x(\locSurv)>0$) for all $x\in D$.}
\end{split}
\end{equation}

We associate to the branch rate $r(x)$ and offspring distribution $(p_n(x):0\leq n\leq n_{\max})$ the nonlinear term
\begin{equation}\label{eq:nonlinearity f}
f(x,u):=r(x)\sum_{n=0}^{n_{\max}}p_n(x)[(1-u)-(1-u)^n].
\end{equation}
For instance associated to binary branching at rate $1$ (i.e. $r\equiv 1$, $p_2\equiv 1$, $p_n\equiv 0$ for $n\neq 2$) \commentout{\BS
i.e. $r=1$, $p_i=1$ for $i=2$ and $=0$ otherwise, right?
But then $v=1-u$ satisfies KPP, not $u$} we have the classical FKPP nonlinearity $u-u^2$.
Then Assumption \ref{assum:standing prob proof} ensures that the domain $D$, reaction term given by \eqref{eq:nonlinearity f}, and $a(x)\equiv\text{Id}$ satisfy Assumption \ref{assum:standing} (with the condition $f\in C^1(D\times \Rm_{\geq 0})$ replaced as in Remark \ref{rmk:generalising f Holder}).

Our proof will hinge on the connection (see \cite[Section 4]{Maillard2025}) between the survival of BBM and steady states of the FKPP equation, 
\Fi{s-a prob proof}
\left\{\begin{array}{ll}
\Delta u+f(x,u)=0, & x\in D,\\
u(x)\overset{\xi}{=}0, & x\in\partial D,
\end{array}\right.
\Ff 
where $f$ is given by \eqref{eq:nonlinearity f}. We now recall this connection, \cite[Theorems 4.15 and 4.17]{Maillard2025}.
\begin{theo}[Theorem 4.15 and 4.17 of \cite{Maillard2025}]\label{theo:PDE survival connection}
We assume Assumption \ref{assum:standing prob proof}. Then the following hold.
\begin{enumerate}
\item 
We have $\Pm_x(\glSurv)>0$ for some (hence all) $x\in D$ if and only if there exists a positive and bounded solution of \eqref{s-a prob proof}, where $f$ is given by \eqref{eq:nonlinearity f}. In the case that this exists then
\[
v(x):=\Pm_x(\glSurv)
\]
is the maximal positive and bounded solution of \eqref{s-a prob proof}, in the sense that $\tilde{u}(x)\leq v(x)$ for any other positive and bounded solution $\tilde{u}$ of \eqref{s-a prob proof}.
\item
We now suppose that $\Pm_x(\locSurv)>0$ for some (hence all) $x\in D$. Then
\[
u(x):=\Pm_x(\locSurv)
\]
is a positive and bounded solution of \eqref{s-a prob proof}, which is minimal in the sense that
\[
u(x)\leq \tilde{u}(x)
\]
for any other positive and bounded solution $\tilde{u}$ of \eqref{s-a prob proof}.
\end{enumerate}
\end{theo}
\begin{proof}[Proof of Theorem \ref{theo:basic facts survival}]
We firstly observe that Theorem \ref{theo:PDE survival connection} combined with the strong maximum principle implies that either $\Pm_x(\locSurv)<\Pm_x(\glSurv)$ everywhere, or $\Pm_x(\locSurv)=\Pm_x(\glSurv)$ everywhere. Combined with similar facts we have already seen about $\Pm_x(\glSurv)$ and $\Pm_x(\locSurv)$, this completes the proof of Part \ref{enum:strict pve or 0 everywhere}.

We recall from \eqref{eq:L prob} that we define
\[
\L=\Delta+c\quad\text{where}\quad c(x):=r(x)\sum_{n=0}^{n_{\max}}p_n(x)(n-1).
\]
Note that $c(x)$ is $f_s(x,0)$. We employ the following generalised principal eigenvalues\footnote{Recall again that opposing sign conventions for $\lambda$ are employed in PDE and probability, so that whereas the following is the definition common in PDEs, which we use here, the typical definition employed in probability differs from these by a factor of $-1$.}
\[
\begin{split}
&\lambda(-\L):=\sup\{\lambda\in \Rm:\exists u\in W^{2,d}_{\text{loc}}(D)\text{ such that $-\L u\geq \lambda u$, $u(x)>0$ for all $x\in D$}\},\\
&\lambda'(-\L):=\inf\{\lambda\in\Rm:\exists u\in W^{2,d}_{\text{loc}}(D)\cap L^{\infty}(D)\text{ such that $-\L u \leq \lambda u,$}\\&\text{$u(x)>0$ for all $x\in D$, $u(x)\overset{\xi}{=} 0,$ $x\in \partial D$.}\}.
\end{split}
\]

The latter eigenvalue was introduced by Berestycki and the first author in \cite{BR4}, along with a third generalised principal eigenvalue, $\lambda''(-\L)$. It was shown that $\lambda=\lambda$ (in the present self-adjoint setting), under stronger ($C^{1,1}$) assumptions on the boundary. Recently, Maillard and the second author established that $\lambda(-\L)=\lambda'(-\L)=\lambda''(-\L)$ \cite[Theorem 4.14]{Maillard2025}, the proof of which clearly works without imposing assumptions on the boundary. Combined with \cite[Theorem 3]{Englander2004} and \cite[Theorems 4.5]{Maillard2025}, Theorem \ref{theo:basic facts survival} is proven.
\end{proof}

\begin{proof}[Proof of Theorem \ref{theo:BBM no glob surv without loc surv}]
We recall that we also assume in Theorem \ref{theo:BBM no glob surv without loc surv} that $ \inf_{x\in D}r(x)\sum_{n=2}^{n_{\max}}p_n(x)>0$. This implies that $f$ satisfies \eqref{eq:existence condition} and \eqref{concav1} with $\ell=1$. 

Since $d\leq 6=2+4/\ell$, Theorem \ref{theo:nonexistence at criticality} implies the non-existence of solutions to \eqref{s-a prob proof} when $\lambda(-\L)\geq 0$, hence $\Pm_x(\glSurv\setminus \locSurv)\leq \Pm_x(\glSurv)\equiv 0$ by Theorem \ref{theo:PDE survival connection}.

Similarly, Theorem \ref{theo:uniqueness} implies the uniqueness of solutions to \eqref{s-a prob proof}, hence for $\lambda(-\L)<0$, $\Pm_x(\glSurv\setminus\locSurv)=\Pm_x(\glSurv)-\Pm_x(\locSurv)\equiv 0$ by Theorem \ref{theo:PDE survival connection}.
\end{proof}

\begin{proof}[Proof of Theorem \ref{theo:counterexample prob}]
This follows immediately by combining Theorems \ref{theo:basic facts survival}, \ref{theo:bound on solns l=1}, \ref{thm:Hardy_spercritical} and \ref{theo:PDE survival connection}, checking that the construction in the proof of Statement \eqref{enum:evalue<=0 and both possible} of Theorem \ref{thm:Hardy_spercritical} still gives $\lambda(-\L)=\lambda(-(\Delta+1-\max(1-V,0)))<0$ with the `max'.
\end{proof}
\begin{proof}[Proof of Theorem \ref{theo:gl surv loc surv inhom D}]
This follows immediately by combining Theorems \ref{theo:existence nonuniq d geq 7} and~\ref{theo:PDE survival connection}.\end{proof}

\section{Homogeneous counterexample}\label{section:homogeneous counterexample}

We now prove Theorem \ref{theo:existence nonuniq d geq 7}, that is we  establish the existence of a domain on which we have non-uniqueness for the (homogeneous) classical FKPP nonlinearity $u-u^2$. We will do this by the probabilistic method: we will construct a random domain which has multiple positive and bounded FKPP steady states with positive probability. The probabilistic connection in Section \ref{section:probabilistic proofs} will be essential here.

Given a ($\sqrt{2}$-diffusivity) Brownian motion $(X_t)_{t\geq 0}$ and $r>0$, the Wiener sausage is the $r$-neighbourhood of $(X_t)_{t\geq 0}$. A classical theorem of Kesten, Spitzer and Whitman (see \cite{LeGall1988}) implies that, for any $d\geq 3$, there exists a constant $\kappa>0$ depending only on the dimension such that, for any $t\in \Rm_{>0}$, the Wiener sausage satisfies
\begin{equation}\label{eq:Wiener sausage convergence}
\text{Vol}[B((X_s:0\leq s\leq t),r)]/r^2\ra \kappa\quad \text{in probability as $r\ra 0$.}
\end{equation}

We now suppose that $d\geq 7$, so that $2d-8<H_d$. Given positive and increasing $R_n\uparrow \infty$ to be determined, we define $\Gamma_0:=B(0,R_0)$ and $\Gamma_n:=B(0,R_n)\setminus B(0,R_{n-1})$ for $n\geq 1$. We now fix $2d-8<c_1<c_2<H_d$ and sequences $(R_n)$ and $(a_n)$ such that
\[
\min(c_1\lvert x\rvert^{-2},1)\leq  \sum_{n\geq 0}a_n\Ind_{\Gamma_n}(x)\leq \min(c_2\lvert x\rvert^{-2},1).
\]
Given a positive sequence $r_n\ra 0$ to be determined, we take a Poisson point process $\{\chi_i:i\geq 1\}$ of intensity
\[
V(x):=\sum_{n\geq 0}\Ind_{\Gamma_n}(x)\frac{(1-a_n)}{\kappa r_n^2}.
\]
We define the random set $D$ to be the unique unbounded connected component of 
\[
E:=\left[\cup_{n\geq 0}\cup_{\chi_i\in \Gamma_n}B(\chi_i,r_n^2)\right]^c,
\]
(which we will show is unique for suitable $(r_n)$). Since no two balls are tangential with probability $1$, $\partial E$ is almost surely locally Lipschitz, hence we can interpret Dirichlet boundary conditions on $\partial E$ in the classical way, almost surely. Precisely, in our construction $E$ will almost surely have a unique unbounded connected component, $D$, and we consider positive and bounded solutions of
\begin{equation}\label{eq:stationary FKPP random D}
    \Delta u+u-u^2-0,\quad \text{on $D$},\quad u\equiv 0\quad\text{on $\partial D$.}
\end{equation}
\begin{theo}\label{theo:random solution}
Suppose that $d\geq 7$. Then there exists a positive sequence $r_n\ra 0$ such that:
\begin{enumerate}
    \item with probability $1$, $E$ has a unique unbounded connected componet, $D$;
    \item $\lambda(-(\Delta+1);D)\leq 0$ and $\Vol(D^c)<\infty$, with probability $1$;
    \item with probability $1$, either there exists a solution of \eqref{eq:stationary FKPP random D} and $\lambda(-(\Delta+1);D)= 0$, or there are at least two solutions of \eqref{eq:stationary FKPP random D} (or both);\label{enum:drift to infinity}
    \item every pair of balls are a positive distance apart (so that $D=E$ and $\partial D$ is locally analytic) and $\lambda(-(\Delta+1);D)<0$, with positive probability.
\end{enumerate}
\end{theo}
This implies the existence of a connected domain with locally analytic boundary supporting multiple positive and bounded solutions of \eqref{eq:stationary FKPP random D}.
\begin{proof}[Proof of Theorem \ref{theo:random solution}]
We fix $\epsilon>0$. We will couple (1) a BBM with constant branching rate $1$ and position-dependent killing at rate $1-V(x)$, started at $0$, $X_t=(X^1_t,\ldots,X^{N_t}_t)$; (2) a BBM with branching rate $1$ and no killing, started at $0$, $Y_t=(Y^1_t,\ldots,Y^{M_t}_t)$; and a random domain $D$ given by a Poission point process as above. Our coupling will be such that (A) $Y_t$ and $D$ are independent, and (B) outside of an event of probability at most $\epsilon$, $Y_t$ survives globally on $D$ without surviving locally if and only if $X_t$ does the same on $\Rm^d$.

\textit{The key idea is to couple the Poisson point process of killing times of $X_t$ with the Poisson point process of balls in the construction of $D$.}

We will assume throughout that $r_n<(R_{n-1}-R_{n-2})/10$ for all $n\geq 2$, hence $B(x,r_m)\cap \Gamma_n=\emptyset$ for $x\in \Gamma_m$ and $n+2\leq m$. 

We now need to introduce some notation. We will construct $X_t$ and $Y_t$ one lifetime at a time, a lifetime being a Brownian trajectory between a particle being born and it branching. The particle lifetimes for $Y_t$ are indexed by $\Sigma:=\cup_{m\geq 0}\{1,2\}^m$, and for $X_t$ are indexed by the random subset $\Tm\subseteq \Sigma$ (so $\Tm$ correspond to those particles which are ever born), using Ulam-Harris notation. We order $\Sigma$ in shortlex order, so given $x,y\in \Sigma$, $x<y$ if $\lvert x\rvert <\lvert y\rvert$ or $\lvert x\rvert =\lvert y\rvert$ and $x$ proceeds $y$ lexicographically (i.e. $\emptyset,(1),(2),(11),(12),(21),(22),(111),\ldots$).

We start with $\Tm=\{\emptyset\}$, which we will inductively build. We let $C\subseteq \cup_{m\geq 1}\{1,2\}^m$ be those vertices yet to consider. We proceed inductively as follows. Given $\Tm$ and $C$, we take $v\in C$ to be the least element of $C$ in shortlex order, and add $v$ to $\Tm$. We inductively observe from the following construction that $\Tm$ is now $\{v\}\cup\{w< v:X^w\text{ is ever born}\}$. 

We take a Brownian motion $X_s^v$, started at time $T^v_0$ until the branch time $T^v_1$. This is the lifetime of $Y^v$, which shares the same path. We then define
\begin{align*}
&A_{n,v}:=B((X^0_s:T_0^v\leq s<T_1^v),r_n)\cap \Gamma_n\setminus \cup_{w<v}A_{n,w},\quad    I_{n,v}:=\{s:X^v_s\in \Gamma_n\},\\
&m_{n,v}:=\min\left(\frac{\text{Vol}(A_v^n)}{\kappa r_n^2},a_n\lvert I^v_n\rvert\right),\quad \delta_{n,v}:=\frac{\text{Vol}(A_{n,v})}{\kappa r_n^2}-a_n\lvert I^v_n\rvert
\end{align*}
and write $\delta^{n,v}_{\pm}:=(\delta_{n,v})_{\pm}$ for the positive and negative parts of $\delta_{n,v}$. 

We draw independently from a $\text{Poisson}(m_{n,v})$ random variable, $Z_{n,v}$ say, and (1) choose $Z_{n,v}$ points independently and uniformly at random from $A_{n,v}$, $\{\chi_i^{n,v}:i\geq 1\}$ say; and (2) choose $Z_{n,v}$ `killing times' from $I_{n,v}$ independently and uniformly at random.

If $\delta^{n,v}_+>0$, we add to $\{\chi^{n,v}\}$ an additional Poisson point process of intensity $\delta^{n,v}_+/\lvert A_{n,v}\rvert$ on $A_{n,v}$. We have now constructed for each $n$ the Poisson point process $\{\chi^{n,v}_i:i\geq 1\}$, and further define 
\[
K_{n,v}:=\cup_{n\geq 0}\cup_{i\geq 1}\bar B(\chi^{n,v}_i,r_n),\quad K_v:=\cup_{n\geq 0}K_{n,v}.
\]
Conversely, if $\delta^{n,v}_->0$, we add an additional Poisson point process of `killing times' of intensity $\delta^{n,v}_-/\lvert I_{n,v}\rvert$ on $I_{n,v}$. We have now constructed all `killing times' on $I_{n,v}$, for all $n\geq 0$.

If there are any additional points added in either case, we say the coupling is `bad'. We denote this event by $\mathbb{B}^1_{n,v}$.

If there is some killing time for $v$, for some $n$, then the first such time is defined to be the killing time of $X^v$, $\tau_v$. In this case, no child of $v$ is ever born in $X^v$, hence we remove all descendents of $v$ from $C$, and the part of $Y_t$ descended from $Y^v$ is given by an independent Brownian motion started from time $T_1^v$ at position $Y^v_{T_1^v}$. Conversely, if there is no killing time for $X^v$, then it has the same lifetime as $Y^v$, and we keep $(v1)$ and $(v2)$ in $C$.

We also say the coupling is `bad' if the path of $Y^v$ hits $K_{w}$ for some $w<v$. If the coupling isn't bad, we say it's `good'. We denote the event that the path of $Y^v$ hits $K_w\cap \Gamma_n$ as $\mathbb{B}^2_{n,v,w}$.

The key observation is that, \textit{if the coupling is good, $X^v$ survives to its branching time if and only if $Y^v$ avoids hitting $K_v$}. We repeat this construction inductively until $C=\emptyset$, at which point we have finalised the construction of $\Tm$, and say the whole coupling is `good' if it's good at each step. We denote this by $\mathbb{G}$, and its complement as $\mathbb{B}=\mathbb{G}^c$. Finally, we take a Poisson process of points on $(\cup_{n\geq 0,v\in \Tm}A_{n,v})^c$ of intensity $V(x)$, $\{\chi^c_i:i\geq 1\}$; note $Y_t$ is never within a distance $r_n$ of $\chi^c_i\in \Gamma_n$. We then define 
\[
\{\chi_i\}:=\{\chi^c_i:i\geq 1\}\cup_{n\geq 0,v\in \Tm}\{\chi^{n,v}_i\},\quad E:=\left(\cup_{n\geq 0,v\in \Tm}\cup_i\bar B(\chi^{n,v}_i,r_n)\right)^c.
\]

By construction, we have (A) $Y_t$ and $E$ are independent, and (B) on $\mathbb{G}$, $Y_t$ survives globally on $E$ without surviving locally if and only if $X_t$ does the same on $\Rm^d$. Our goal now is to show $(r_n)$ can be taken ensuring the coupling is good with high probability.

We now fix $\epsilon>0$ to be determined. Since we have almost sure local extinction for $X^v$, there exists $\ell(n)$ increasing in $n$ with $\ell(n)\ra \infty$ as $n\ra\infty$, such that
\[
\Pm(\mathbb{D}_n):=\Pm(\text{$X^v$ or $Y^v$ enters $\bar B_{R_n}$ for some $v\in \Tm$, $\lvert v\rvert> \ell(n)$})\leq \epsilon 2^{-n}\quad\text{for $n\geq 0$.}
\]
\commentout{It also implies that there exists $L_0(\ell)\ra \infty$ as $\ell\ra \infty$ such that 
\[
\Pm(Y^w\in B(0,R_{L_0(\ell)})\text{ for some $\lvert v\rvert=\ell$})<\epsilon 2^{-\ell}$.}

Putting all this together,
\begin{align*}
&\mathbb{B}\subseteq \cup_{n\geq 0,\;v\in \Tm}\mathbb{B}^1_{n,v}\cup_{n\geq 0,\; v,w\in \mathbb{T}:w<v}\mathbb{B}^2_{n,v,w}\cup\cup_{n\geq 0}\mathbb{D}_n\\
&\subseteq \cup_{n\geq 0}\cup_{v\in \Tm:\lvert v\rvert\leq \ell(n+1)}\mathbb{B}^1_{n,v}\cup \cup_{n\geq 0}\cup_{v,w\in \Tm:w<v,\lvert v\rvert\leq \ell(n+1)}\mathbb{B}^2_{n,v,w}\cup\cup_{n\geq 0}\mathbb{D}_n.
\end{align*}
where in the second inclusion we are using that $\mathbb{B}^1_{n,v},\mathbb{B}^2_{n,v,w}\subseteq \mathbb{D}_n$ for $\lvert v\rvert\geq \ell(n+1)$. We now write $p(v)$ for the parent of $v$ (for $\lvert v\rvert\geq 1$), and write $G_v$ for the event that the coupling is good for $v$. We observe that, on $\mathbb{B}^2_{n,v,p(v)}$, then $K_{n,w}\neq \emptyset$ but $X^v$ is not killed, hence the coupling is bad for $p(v)$. Therefore
\begin{align*}
&\mathbb{B}\subseteq \cup_{n\geq 0}\cup_{v\in \Tm:\lvert v\rvert\leq \ell(n+1)}\mathbb{B}^1_{n,v}\cup \cup_{n\geq 0}\cup_{v,w\in \Tm:w<v,w\neq p(v),\lvert v\rvert\leq \ell(n+1)}\mathbb{B}^2_{n,v,w}\cup\cup_{n\geq 0}\mathbb{D}_n\\
&:=\cup_{n\geq 0}\mathbb{B}^1_n\cup \cup_{n\geq 0} \mathbb{B}^2_n\cup \cup_{n\geq 0}\mathbb{D}_n.
\end{align*}

Then using a union bound,
\begin{align*}
\Pm(\mathbb{B})\leq \sum_{n\geq 0}\Pm(\mathbb{B}^1_n)+\sum_{n\geq 0}\Pm(\mathbb{B}^2_n)+2\epsilon.
\end{align*}

We assume, inductively in $m$, that we have a positive decreasing sequence $(\bar r_n)$ such that $\Pm(\mathbb{B}^1_{n}),\Pm(\mathbb{B}^2_{n})<\epsilon 2^{-n}$ for $n\leq m$ whenever $r_n\leq \bar r_n$ for all $n$. Then by decreasing only $\bar r_{m},\bar r_{m+1},\bar r_{m+2}$ we ensure this for $ \Pm(\mathbb{B}^1_{m+1})$ by \eqref{eq:Wiener sausage convergence}, and for $\Pm(\mathbb{B}^2_{m+1})$ since a Brownian motion started away from a second Brownian motion almost surely won't hit it in $d\geq 4$ \cite[Theorem 9.1, p.255]{Morters2010} (note that this is why we worked to exclude the parent of $v$ from the previous union). The proof for the base case $n=0$ is identical.

\commentout{

\begin{align*}
&\Pm(\text{coupling is good})\leq \sum_n\mathcal{E}_n+\sum_{\ell}\mathcal{F}_\ell\\
&:=\sum_n\expE\left[\epsilon 2^{-n}+\sum_{v\in \Tm,\lvert v\rvert\leq \ell(n)}\lvert \delta^v_{n}\rvert\right]\\
&+\expE\left[\sum_{\ell\geq 0}\sum_{n\geq 0}\Ind(Y^v\text{ hits $K_w\cap \Gamma_n$ for some $w<v$, $w\in \Tm$})\right].
\end{align*}
We assume, inductively in $m$, that we have chosen $(\bar r_n)$ such that $\mathcal{E}_n\leq 2\epsilon 2^{-n}$ for $n\leq m$ and $\mathcal{F}_{\ell}\leq \epsilon 2^{-\ell}$ for $\ell\leq m$. Since $\delta^v_{m+1}$ only depends on $r_{m+1}$, and $\sum_{\lvert v\rvert\leq \ell(m+1)}\delta^v_{m+1}\ra 0$ in probability as $r_{m+1}$ independently of the other $r_n$s, decreasing $r_{m+1}$ while keeping the other $r_n$s constant ensures that we have $\mathcal{E}_{n+1}\leq 2\epsilon 2^{-(n+1)}$. Similarly, keeping $r_0,\ldots,r_{L_0(m+1)-3}$ fixed but decreasing the other $r_n$ ensures that $\mathcal{F}_{m+1}\leq 2\epsilon 2^{-(m+1)}$. 

Repeating this process inductively, we get $(r_n)$ such that $\Pm(\text{coupling is good})\leq 4\epsilon$, and this remains the case by any further decreasing of the $r_n$s. 

}

Therefore, by induction, we can take a positive and decreasing sequence $(\bar r_n)$ ensuring that $\Pm(\mathbb{B})\leq 6\epsilon$ whenever $0<r_n<\bar r_n$ for all $n\geq 0$. It is easy to see that the $r_n$s can then be chosen sufficiently small to ensure that $\text{Vol}(E^c)<\infty$ and the number of balls that overlap is finite almost surely with probability $1$ (hence there's a unique connected component of $E$). We henceforth fix such a choice of positive, decreasing $r_n\ra 0$. There is then a positive probability that there are no overlaps (so that $D=E$) and no balls in $B(0,R)$ for $R$ large enough, ensuring that $\lambda(-(\Delta+1;D))=\lambda(-(\Delta+1;E))\leq  \lambda(-(\Delta+1);B(0,R))<0$.

We have left only to check Part \ref{enum:drift to infinity}. Since $\{\chi_i:i\geq 1\}$ and $Y_t$ are independent, we can place them onto separate probability spaces $(\Omega,\mathcal{F},\Pm)$ and $(\tilde{\Omega},\tilde{\mathcal{F}},\tilde{\Pm})$ respectively. Therefore $D$ is $\mathcal{F}$-measurable. We further consider the filtration $(\mathcal{F}_R:R\geq 0)$ defined by $\mathcal{F}_R:=\sigma(\chi_i:\chi_i\in \bar B(0,R))$. We define the event 
\[
F:=\{\omega\in \Omega:\tilde{\Pm}_x(\glSurv\setminus\locSurv)>0\text{ for $x\in D(\omega)$}\}\in \mathcal{F}, 
\]
where $\glSurv\setminus\locSurv$ is the event that $Y_t$ on $D(\omega)$ survives globally but not locally, if considered to be killed on $\partial D$. This is possible for $x\in D$ if and only if it's possible on $D\setminus B(0,R)$ for some $x\in D\setminus B(0,R)$, hence $F$ belongs to the tail $\sigma$-algebra of $(\mathcal{F}_R:R\geq 0)$. Kolmogorov's zero-one law then implies $\Pm(F)\in \{0,1\}$. However, we have previously established that $\Pm\otimes \tilde{\Pm}(\glSurv\setminus \locSurv)>0$, so $\Pm(F)>0$ hence $\Pm(F)=1$. Therefore we have Part \ref{enum:drift to infinity}. The theorem is proven.\end{proof}

\subsection*{Acknowledgements}

This collaboration was initiated at the workshop `Emerging Connections between Reaction-Diffusion, Branching Processes, and Biology' (BIRS 25w5453) at the Banff International Research Station, 18--23 May 2025. We are grateful to the organisers and BIRS for their hospitality. 
This work has received funding from Excellence Project of Sapienza - Dipartimento di Matematica
(CUP B83C23001390001), and from the European Union -- Next Generation EU, PRIN project 2022W58BJ5 ``PDEs and optimal control methods in mean field games, population dynamics and multi-agent models''.

\bibliography{UniquenessFKPP.bib}
\bibliographystyle{plain}

\commentout{

}

\end{document}